\documentclass{article}
\usepackage[utf8]{inputenc}
\usepackage{amssymb, mathtools, hyperref, amsthm, authblk, xcolor, cancel}

\usepackage[backend=biber]{biblatex}
\AtEveryBibitem{\clearlist{language}}
\numberwithin{equation}{section}
\newtheorem{thm}{Theorem}[section]
\newtheorem{prop}[thm]{Proposition}
\newtheorem{lem}[thm]{Lemma}
\newtheorem{cor}[thm]{Corollary}
\newtheorem{rmk}[thm]{Remark}
\theoremstyle{definition}
\newtheorem{defn}[thm]{Definition}

\newcommand{\nat}{\mathbb{N}}
\newcommand{\real}{\mathbb{R}}
\newcommand{\complex}{\mathbb{C}}

\newcommand{\iu}{\mathrm{i}}
\newcommand{\eps}{\varepsilon}

\newcommand{\I}{\mathcal{I}}
\newcommand{\J}{\mathcal{J}}

\newcommand{\op}{- \Delta_\alpha}

\newcommand{\dif}{\mathrm{d}}

\DeclareMathOperator{\Dom}{Dom}

\DeclarePairedDelimiter{\abs}{\lvert}{\rvert}
\DeclarePairedDelimiter{\norm}{\lVert}{\rVert}
\DeclarePairedDelimiter{\parens}{(}{)}
\DeclarePairedDelimiter{\set}{\{}{\}}
\DeclarePairedDelimiter{\brackets}{\lbrack}{\rbrack}

\DeclarePairedDelimiter{\angles}{\langle}{\rangle}
\DeclarePairedDelimiter{\cci}{\lbrack}{\rbrack}
\DeclarePairedDelimiter{\coi}{\lbrack}{\lbrack}
\DeclarePairedDelimiter{\oci}{\rbrack}{\rbrack}
\DeclarePairedDelimiter{\ooi}{\rbrack}{\lbrack}

\title{Minimizers of mass-constrained functionals involving a nonattractive point interaction}
\begin{document}

\author{Gustavo de Paula Ramos\thanks{gustavopramos@gmail.com}}
\affil{Instituto de Matemática e Estatística, Universidade de São Paulo, Rua do Matão, 1010, 05508-090 São Paulo SP, Brazil}

\date{\today}
\maketitle
\begin{abstract}
We establish conditions to ensure the existence of minimizer for a class of mass-constrained functionals involving a nonattractive point interaction in three dimensions. The existence of minimizers follows from the compactness of minimizing sequences which holds when we can simultaneously rule out the possibilities of vanishing and dichotomy. The proposed method is derived from the strategy used to avoid vanishing in Adami, Boni, Carlone \& Tentarelli (Calc. Var. 61, 195 (2022)) and the strategy used to avoid dichotomy in Bellazzini \& Siciliano (J. Funct. Anal. 261, 9 (2011)). As applications, we prove the existence of ground states with sufficiently small mass for the following nonlinear problems with a point interaction: a Kirchhoff-type equation and the Schrödinger--Poisson system.

\smallskip
\sloppy \noindent \textbf{Keywords.} Constrained minimization, delta interaction, standing waves, Kirchhoff equation, Schrödinger--Poisson system.
\end{abstract}

\tableofcontents

\section{Introduction}
\subsection{Motivation and goals}
\label{intro:motivation}

The existence of critical points of constrained functionals is a recurrent problem in both mathematics and physics. An important example of this kind of problem is the search for standing waves of Schrödinger-type equations. Indeed, the mass of solutions to these equations is constant with respect to time evolution, so there is a correspondence between standing waves and critical points of the associated mass-constrained energy functional.

This paper is concerned with functionals involving a point interaction in a sense that we will precise in what follows. In physics, the presence of defects or impurities is modeled by considering a formal $\delta$ potential. For instance, the following formal equation was recently proposed in \cite{sakaguchiSingularSolitons2020, shamrizSingularMeanfieldStates2020} as a model for the spatial profile of standing waves of the defocusing septimal \emph{Nonlinear Schrödinger Equation (NLSE)} on $\real$ to explain the existence of singular solitons:
\[- \frac{1}{2} u'' - \alpha \delta_0 u + \omega u = - u^7,\]
where $0 < \alpha, \omega < \infty$.

Instead of the formal operator $- \Delta + \frac{1}{\alpha} \delta_0$, we consider the mathematically rigorous model for point interactions given by the \emph{Hamiltonian of point interaction}
\[\op \colon \Dom \parens{\op} \to L^2\]
as defined in Section \ref{intro:nonlinear}. While there are many papers about nonlinear problems involving $\op$ on
$\real$ (for instance, \cite{adamiConstrainedEnergyMinimization2013, adamiExistenceDynamics1D2009, angulopavaStabilityPropertiesCubicQuintic2019, boniPrescribedMassGround2021, fukuizumiNonlinearSchrodingerEquation2008, lecozInstabilityBoundStates2008, pavaNonlinearSchrodingerEquation2013, pavaStabilityStandingWaves2017}), similar problems on higher dimensions have only started being considered recently.

Let us briefly review recent developments about nonlinear problems involving a point interaction in higher dimensions. First, consider the \emph{NLSE with a point interaction} ($\delta$-NLSE). The existence of action minimizers in dimension 2 was addressed in \cite{fukayaStabilityInstabilityStanding2022}. In \cite{adamiExistenceStructureRobustness2022, adamiGroundStatesPlanar2022}, Adami, Boni, Carlone \& Tentarelli established the existence and qualitative properties of ground states of the following $\delta$-NLSE in $\real^n$ for $n \in \set{2, 3}$:
\begin{equation}
\label{intro:eqn:NLSE}
\iu \partial_t \psi = \op \psi - \psi \abs{\psi}^{p - 2}
\end{equation}
with $\alpha \in \real$, $2 < p < 4$ if $n = 2$ and $2 < p < 3$ if
$n = 3$. In \cite{cacciapuotiWellPosednessNonlinear2021}, Cacciapuoti, Finco \& Noja addressed the local and global well-posedness of the following Cauchy problem in $\real^n$ for $n \in \set{2, 3}$:
\[
\begin{cases}
\iu \partial_t \psi = \op \psi \pm \psi \abs{\psi}^{p - 1};
\\
\psi \parens{0} = \psi_0,
\end{cases}
\]
where $\alpha \in \real$, $p \geq 1$ if $n = 2$ and $1 \leq p < 3 / 2$ if $n = 3$. This problem was also considered in \cite{cacciapuotiFailureScatteringNLSE2023}, where failure of scattering was established for $1 < p < 2$ if $n = 2$ and
$1 < p < 4 / 3$ if $n = 3$. To finish, we remark that there are also a few results about the Hartree equation with a point interaction (see \cite{michelangeliSingularHartreeEquation2021, georgievStandingWavesGlobal2024}).

In this context, our goal is to propose conditions that ensure the existence of minimizers for mass-constrained functionals involving a \emph{nonattractive} point interaction in $\real^3$, that is, $\alpha \in \coi{0, \infty}$. As we can only precisely state the general form of the considered minimization problem after defining the energy space $W^{1, 2}_\alpha$ in Section \ref{intro:energy-space}, we remark that we are interested in functionals which are often obtained as the energy functional associated to variational equations in $\real^3$ of the form
\[- f_1 \parens{u} \Delta_\alpha u + \omega u + f_2 \parens{u} = 0,\]
where $f_1$, $f_2 \colon \Dom \parens{\op} \to L^2$ are (possibly nonlocal) nonlinearities and we want to solve for
$\omega \in \real$, $u \in \Dom \parens{\op}$.

Let us comment on the organization of the rest of the introduction.
\begin{itemize}
\item
In Section \ref{intro:notation}, we fix the notation used throughout the paper.
\item
In Section \ref{intro:nonlinear}, we recall the precise definition of $\op$.
\item
In Section \ref{intro:energy-space}, we define the Hilbert space $W^{1, 2}_\alpha$ on which functionals involving a point interaction are defined.
\item
In Section \ref{intro:AMP}, we state the abstract problem considered in this paper.
\item
In Section \ref{intro:existence-of-soln}, we state our results about existence of solutions to abstract problems as in Section \ref{intro:AMP}.
\item
In Section \ref{intro:applications}, we state our existence results for ground states of concrete nonlinear problems involving a point interaction.
\end{itemize}

\subsection{Notation}
\label{intro:notation}

\begin{itemize}
\item
Unless mentioned otherwise, we implicitly suppose that $\alpha \in \coi{0, \infty}$.
\item
If $A, B$ are sets, $f \colon A \to B$ is a function and $A'$ is a subset of $A$, then $f|_{A'}$ denotes the \emph{restriction} of $f$ to $A'$, that is, the function
\[A' \ni a' \mapsto f \parens{a'} \in B.\]
\item
Suppose that $X$ and $Y$ are complex Banach spaces. In this context, we employ the notation summarized below:
\begin{itemize}
\item
We let $\mathcal{L} \parens{X, \real}$ denote the real Banach space of continuous
$\real$-linear functionals from $X$ to $\real$.
\item
If $T \colon X \to \real$ is $\real$-linear, then we denote its value at $x \in X$ by
$T \brackets{x} \in \real$.
\item
We write $X \hookrightarrow Y$ to mean that $X$ is canonically continuously embedded in $Y$.
\end{itemize}
\item
If $H$ is a Hilbert space, then its inner product is linear with respect to its second entry and we denote it by
$\parens{h_1, h_2} \mapsto \angles{h_1 \mid h_2}_H$.
\item
\sloppy
Unless mentioned otherwise, we integrate in $\real^3$ with respect to the Lebesgue measure and we consider functional spaces of complex-valued functions defined a.e. in
$\real^3$.
\item
The \emph{homogeneous Sobolev space} $\dot{W}^{1, 2}$ is defined as the Hilbert space completion of $C_c^\infty$ with respect to
$
\angles{u \mid v}_{\dot{W}^{1, 2}}
:=
\int \parens{\nabla \overline{u} \cdot \nabla v}
$.
\item
The \emph{Sobolev space} $W^{1, 2}$ is defined as the Hilbert space completion of $C_c^\infty$ with respect to
$
\angles{u \mid v}_{W^{1, 2}}
:=
\angles{u \mid v}_{\dot{W}^{1, 2}} + \angles{u \mid v}_{L^2}
$.
\item
The \emph{Sobolev space} $W^{2, 2}$ is defined as the Hilbert space completion of $C_c^\infty$ with respect to
\[
\angles{u \mid v}_{W^{2, 2}}
:=
\parens*{
	\sum_{1 \leq j, k \leq 3}
	\int
		\overline{\partial_j \partial_k u \parens{x}}
		\partial_j \partial_k v \parens{x}
	\dif x
}
+
\angles{u \mid v}_{W^{1, 2}}.
\]
\item
Throughout the paper, we implicitly use the fact that
$W^{2, 2} \hookrightarrow C^{0, \frac{1}{2}}$ that follows from Morrey's Embedding (see \cite[Theorem 12.55]{leoniFirstCourseSobolev2017}).
\end{itemize}

\subsection{The operator $\op$}
\label{intro:nonlinear}

Let us recall the definition of $\op$ according to the approach in \cite[Section I.1.1]{albeverioSolvableModelsQuantum1988}. The closure of
$- \Delta|_{C_0^\infty \parens{\real^3 \setminus \set{0}}}$ in $L^2$
admits a family of $L^2$-self-adjoint extensions, which we denote by
$\set{\op}_{\alpha \in \oci{- \infty, \infty}}$. The case $\alpha = \infty$ corresponds to the Friedrichs extension:
\[
- \Delta_\infty \phi = - \Delta \phi
\quad \text{for every} \quad
\phi \in \Dom \parens{- \Delta_\infty} = W^{2, 2}.
\]

Now, consider the case $\alpha \in \real$. We need to introduce a family of Green's functions to describe $\op$. Given $\lambda \in \ooi{0, \infty}$, we define
$G_\lambda \colon \real^3 \setminus \set{0} \to \ooi{0, \infty}$
as
\[
G_\lambda \parens{x}
=
\frac{e^{- \sqrt{\lambda} \abs{x}}}{4 \pi \abs{x}},
\]
so that
\begin{equation}
\label{eqn:Green's-fct}
- \Delta G_\lambda + \lambda G_\lambda = \delta_0
\quad \text{in the sense of distributions.}
\end{equation}
Standard integral calculus shows that
\begin{equation}
\label{eqn:integrability-of-G_lambda}
G_\lambda \in L^r \quad \text{if, and only if,} \quad 1 \leq r < 3.
\end{equation}
More precisely, $\norm{G_\lambda}_{L^2}^2 = \frac{1}{8 \pi \sqrt{\lambda}}$ and
\[
\norm{G_\lambda}_{L^r}^r
=
\frac{\norm{G_1}_{L^r}^r}{\lambda^{\frac{3 - r}{2}}}
\quad \text{for every} \quad
r \in \coi{1, 3}.
\]

In this context, the operator $\op$ acts as
\[
\op u = - \Delta \phi_\lambda - q \lambda G_\lambda
\quad \text{for every} \quad
u =  \phi_\lambda + q G_\lambda \in \Dom \parens{\op},
\]
where
\begin{multline*}
\Dom \parens{\op}
:=
\left\{
	\phi_\lambda + q G_\lambda
	:
	\parens{\lambda, \phi_\lambda, q}
	\in
	\ooi{0, \infty} \times W^{2, 2} \times \complex,
\right.
\\
\left.
	\alpha + \frac{\sqrt{\lambda}}{4 \pi} \neq 0
	\quad \text{and} \quad
	\phi_\lambda \parens{0}
	=
	\parens*{\alpha + \frac{\sqrt{\lambda}}{4 \pi}} q
\right\}
\subset
L^2.
\end{multline*}
An essential aspect of $\Dom \parens{\op}$ is the nonuniqueness of representation of its elements as stated more precisely in the following corollary of Lemma \ref{lem:G_lambda-G_mu-in-W^{2, 2}}.

\begin{cor}
\label{cor:nonuniqueness}
Suppose that $\alpha \in \real$; $\lambda, \mu \in \ooi{0, \infty}$ and $q \in \complex$.
\begin{enumerate}
\item
Given $\phi_\lambda \in W^{1, 2}$, there exists a unique $\phi_\mu \in W^{1, 2}$ such that
\[\phi_\lambda + q G_\lambda = \phi_\mu + q G_\mu.\]
\item
If
$
\parens*{\alpha + \frac{\sqrt{\lambda}}{4 \pi}}
\parens*{\alpha + \frac{\sqrt{\mu}}{4 \pi}}
\neq
0
$
and $\phi_\lambda \in W^{2, 2}$ is such that
\[\phi_\lambda + q G_\lambda \in \Dom \parens{\op},\]
then there exists a unique $\phi_\mu \in W^{2, 2}$
such that
\[\phi_\lambda + q G_\lambda = \phi_\mu + q G_\mu \in \Dom \parens{\op}.\]
\end{enumerate}
\end{cor}

\subsection{The energy space $W^{1, 2}_\alpha$}
\label{intro:energy-space}

The form domain associated with $\op$ is given by
\[
\Dom \brackets{\op}
=
\set*{
	\phi_\lambda + q G_\lambda:
	\parens{\lambda, \phi_\lambda, q} \in \ooi{0, \infty} \times W^{1, 2} \times \complex
}
\subset
L^2
\]
(see \cite[p. 62--63]{galloneSelfAdjointExtensionSchemes2023}). In view of Corollary \ref{cor:nonuniqueness}, functions in $\Dom \brackets{\op}$ can also be written in infinitely different ways in function of $\lambda \in \ooi{0, \infty}$. Given
\[
u = \phi_\lambda + q_u G_\lambda
\in
\Dom \brackets{\op}
\quad \text{and} \quad
v = \psi_\lambda + q_v G_\lambda
\in
\Dom \brackets{\op},
\]
let
\begin{align*}
S_\alpha \parens{u, v}
&=
\angles{\phi_\lambda \mid \psi_\lambda}_{\dot{W}^{1, 2}}
+
\lambda
\parens*{\angles{\phi_\lambda \mid \psi_\lambda}_{L^2} - \angles{u \mid v}_{L^2}}
+
\parens*{\alpha + \frac{\sqrt{\lambda}}{4 \pi}} q_u \overline{q_v};
\\
&=
\angles{\phi_\lambda \mid \psi_\lambda}_{\dot{W}^{1, 2}}
-
\lambda
\angles{\overline{q_v} \phi_\lambda + q_u \overline{\psi_\lambda} \mid G_\lambda}_{L^2}
+
\parens*{\alpha + \frac{\sqrt{\lambda}}{8 \pi}} q_u \overline{q_v}.
\end{align*}
It follows from Lemma \ref{lem:S_alpha-is-well-defined} that this definition gives us a well-defined sesquilinear form
\[S_\alpha \colon \Dom \brackets{\op} \times \Dom \brackets{\op} \to \complex.\]

Let $H_\alpha \colon \Dom \brackets{\op} \to \real$ be defined as
\[
H_\alpha \parens{u}
=
S_\alpha \parens{u, u}
=
\norm{\phi_\lambda}_{\dot{W}^{1, 2}}^2
+
\lambda \parens*{\norm{\phi_\lambda}_{L^2}^2 - \norm{u}_{L^2}^2}
+
\parens*{\alpha + \frac{\sqrt{\lambda}}{4 \pi}} \abs{q}^2
\]
for every $u = \phi_\lambda + q G_\lambda \in \Dom \brackets{\op}$, i.e., $H_\alpha$ denotes the extension of
\[\Dom \parens{\op} \ni u \mapsto \angles{\op u \mid u}_{L^2} \in \real\]
to $\Dom \brackets{\op}$. Notice that
\[
H_\alpha \parens{\phi} = \norm{\phi}_{\dot{W}^{1, 2}}^2
\quad \text{for every} \quad
\phi \in W^{1, 2}.
\]

The following result states a few properties of $S_\alpha$ and $H_\alpha$ that follow from the nonuniqueness of representation of elements in $\Dom \brackets{\op}$.

\begin{lem}
\label{lem:H_alpha}
Suppose that $\alpha \in \real$. It holds that
\[
H_\alpha \parens{u}
=
\norm{\phi}_{\dot{W}^{1, 2}}^2
+
\frac{\abs{q}^4}{\parens{8 \pi \norm{u}_{L^2}}^2} \parens*{
	1 + \frac{\norm{\phi}_{L^2}^2}{\norm{u}_{L^2}^2}
}
+
\alpha \abs{q}^2
\]
for every
\[
u
:=
\phi + q G_{\eps \abs{q}^4 / \norm{u}_{L^2}^4}
\in
\Dom \brackets{\op} \setminus W^{1, 2},
\]
where $\eps := \frac{1}{\parens{8 \pi}^2}$. If we suppose further that
$\alpha \in \coi{0, \infty}$, then
\begin{enumerate}
\item
$H_\alpha \parens{u} \geq 0$ for every $u \in \Dom \brackets{\op}$;
\item
$H_\alpha \parens{u} = 0$ if, and only if, $u \equiv 0$;
\item
$S_\alpha$ is an inner product.
\end{enumerate}
\end{lem}

In this context, we define the \emph{energy space} $W^{1, 2}_\alpha$ as the following inner product space:
\[
W^{1, 2}_\alpha
=
\parens*{\Dom \brackets{\op}, \angles{\cdot \mid \cdot}_{W^{1, 2}_\alpha}},
\]
where
\[
\angles{u \mid v}_{W^{1, 2}_\alpha}
:=
\angles{u \mid v}_{L^2} + S_\alpha \parens{u, v}.
\]
Actually, $W^{1, 2}_\alpha$ is a Hilbert space due to structural properties of lower semi-bounded self-adjoint
extensions of symmetric operators (see \cite[Chapter 2]{galloneSelfAdjointExtensionSchemes2023}).

Let us comment on the embeddings of $W^{1, 2}_\alpha$. It is clear that
\[W^{1, 2} \hookrightarrow W^{1, 2}_\alpha \hookrightarrow L^2.\]
In view of \eqref{eqn:integrability-of-G_lambda}, we deduce that the following Sobolev-type embeddings hold:
\[
W^{1, 2}_\alpha \hookrightarrow L^r
\quad \text{when} \quad
2 \leq r < 3.
\]

We are mostly interested in real-valued nonlinear functionals on $W^{1, 2}_\alpha$, so we introduce a notion of boundedness of derivatives for these functionals as follows.

\begin{defn}
Suppose that $T_\alpha \colon W^{1, 2}_\alpha \to \real$ is differentiable. We say that its derivative,
$T_\alpha' \colon W^{1, 2}_\alpha \to \mathcal{L} \parens{W^{1, 2}_\alpha, \real}$,
is \emph{bounded} when the following implication is satisfied: if $U$ is a bounded subset of $W^{1, 2}_\alpha$, then
$\set{T' \parens{u}}_{u \in U}$ is a bounded subset of
$\mathcal{L} \parens{W^{1, 2}_\alpha, \real}$.
\end{defn}

\subsection{Considered problem}
\label{intro:AMP}

\subsubsection{The abstract minimization problem}
\label{intro:AMP:1}

Suppose that
$T_\alpha \colon W^{1, 2}_\alpha \to \real$ is a nonlinear functional such that
\begin{equation}
\label{intro:eqn:T_1}
T_\alpha \in C^1 \parens{W^{1, 2}_\alpha, \real}
\end{equation}
and
\begin{equation}
\label{intro:eqn:T_2}
T_\alpha' \colon W^{1, 2}_\alpha \to \mathcal{L} \parens{W^{1, 2}_\alpha, \real}
\quad \text{is bounded}.
\end{equation}
Consider the nonlinear functional
$I_\alpha \colon W^{1, 2}_\alpha \to \real$ defined as
\[
I_\alpha \parens{u}
=
\frac{1}{2} H_\alpha \parens{u} + T_\alpha \parens{u}.
\]
The goal of this paper is to establish conditions under which the following \emph{abstract minimization problem} admits a solution:
\begin{equation}
\label{eqn:AMP}
\begin{cases}
I_\alpha \parens{u}
=
\I_\alpha \parens{\rho^2}
:=
\inf_{W^{1, 2}_\alpha \parens{\rho^2}} I_\alpha;
\\
u \in W^{1, 2}_\alpha \parens{\rho^2},
\end{cases}
\end{equation}
where $W^{1, 2}_\alpha \parens{\rho^2} := \set{u \in W^{1, 2}_\alpha: \norm{u}_{L^2}^2 = \rho^2}$.

\subsubsection{An auxiliary minimization problem}

It will be useful to also consider a problem related to \eqref{eqn:AMP}. The nonlinear functional $I_\alpha$ extends the functional $I \colon W^{1, 2} \to \real$ defined as
\[
I \parens{\phi}
=
\frac{1}{2} \norm{\phi}_{\dot{W}^{1, 2}}^2
+
T \parens{\phi},
\]
where $T := T_\alpha|_{W^{1, 2}}$. As such, we can analogously consider the following \emph{auxiliary minimization problem}:
\begin{equation}
\label{eqn:AMP:2}
\begin{cases}
I \parens{\phi}
=
\I \parens{\rho^2}
:=
\inf_{W^{1, 2} \parens{\rho^2}} I; \\
\phi \in W^{1, 2} \parens{\rho^2},
\end{cases}
\end{equation}
where
$
W^{1, 2} \parens{\rho^2}
:=
\set{\phi \in W^{1, 2}: \norm{\phi}_{L^2}^2 = \rho^2}
$.

\subsection{Existence of solution to \eqref{eqn:AMP}}
\label{intro:existence-of-soln}

\subsubsection{How may compactness fail?}
Suppose that $\parens{u_n}_{n \in \nat}$ is a minimizing sequence of
$I_\alpha|_{W^{1, 2}_\alpha \parens{\rho^2}}$ and
$u_n \rightharpoonup u_\infty$ in $W^{1, 2}_\alpha$ as $n \to \infty$.
Due to the embedding $W^{1, 2}_\alpha \hookrightarrow L^2$, it also holds that
$u_n \rightharpoonup u_\infty$ in $L^2$ as $n \to \infty$.
The norm $\norm{\cdot}_{L^2}$ is weakly lower semi-continuous, so
\[\mu := \norm{u_\infty}_{L^2} \in \cci{0, \rho}.\]
We classify how the compactness of
$\parens{u_n}_{n \in \nat}$ may fail in function of the value of $\mu$ as follows:
\begin{itemize}
\item
if $\mu = 0$, then we say that \emph{vanishing} occurs;
\item
if $0 < \mu < \rho$, then we say that \emph{dichotomy} occurs.
\end{itemize}
In this context, our main results list conditions under which we can rule out vanishing, rule out dichotomy and ensure that
$I_\alpha \parens{u_\infty} = \I_\alpha \parens{\rho^2}$, from which we deduce the existence of a solution to \eqref{eqn:AMP}.

\subsubsection{Avoiding vanishing}
The following result contains conditions to rule out vanishing and was developed inspired by the arguments in Adami, Boni, Carlone \& Tentarelli's \cite{adamiExistenceStructureRobustness2022, adamiGroundStatesPlanar2022}.

\begin{lem}
\label{lem:vanishing}
Suppose that $\rho \in \ooi{0, \infty}$ is such that
\begin{equation}
\label{cond:vanishing:1}
\I_\alpha \parens{\rho^2} < \I \parens{\rho^2},
\end{equation}
\begin{equation}
\label{cond:vanishing:2}
I_\alpha|_{W^{1, 2}_\alpha \parens{\rho^2}} ~ \text{is coercive}
\end{equation}
and let $\parens{u_n = \phi_{\lambda, n} + q_n G_\lambda}_{n \in \nat}$ be a minimizing sequence of $I_\alpha|_{W^{1, 2}_\alpha \parens{\rho^2}}$. Then
$\liminf_{n \to \infty} \abs{q_n} > 0$. If we suppose further that
\[
u_n
\xrightharpoonup[n \to \infty]{W^{1, 2}_\alpha}
u_\infty
=
\phi_{\lambda, \infty} + q_\infty G_\lambda,
\]
then $q_\infty \neq 0$.
\end{lem}

It may be difficult to verify that \eqref{cond:vanishing:1} is satisfied. In practice, one is often interested in functionals $T_\alpha$ for which given
$v \in W^{1, 2}_\alpha$, there exists $z_v \in L^2$ such that
\begin{equation}
\label{eqn:form-of-T_alpha'}
T_\alpha' \parens{v} \brackets{u}
=
\Re \brackets*{\int \overline{z_v \parens{x}} u \parens{x} \dif x}
\quad \text{for every} \quad
u \in W^{1, 2}_\alpha,
\end{equation}
which is the case for
$T_\alpha \parens{u} := \frac{1}{p} \norm{u}_{L^p}^p$ with $p \in \coi{2, \frac{5}{2}}$.
In this context, the following result inspired in Boni \& Gallone's \cite[Appendix B]{boniTwoDimensionalNLS2025} gives a sufficient condition that is verified in many problems of practical interest and which implies \eqref{cond:vanishing:1}.

\begin{lem}
\label{lem:calI_alpha<calI}
Suppose that $v \in W^{2, 2}$ solves \eqref{eqn:AMP:2}, $v \parens{0} \neq 0$ and there exists $z_v \in L^2$ such that \eqref{eqn:form-of-T_alpha'} holds. Then \eqref{cond:vanishing:1} is satisfied.
\end{lem}

These lemmas are proved in Section \ref{vanishing}.

\subsubsection{Avoiding dichotomy}
In \cite[Section I.2]{lionsConcentrationcompactnessPrincipleCalculus1984}, Lions proved that under the hypothesis that
$I \colon W^{1, 2} \to \real$ is \emph{invariant by translation} in the sense that
\[
I \parens{u} = I \parens*{u \parens{\cdot + x}}
\quad \text{for every} \quad
u \in W^{1, 2} \quad \text{and} \quad x \in \real^3,
\]
the following equivalence holds: minimizing sequences of
$I|_{W^{1, 2} \parens{r^2}}$ are relatively compact up to translations if, and only if, the \emph{strict sub-additivity condition} 
\[
\I \parens{r^2}
<
\I \parens{\mu^2} + \I \parens{r^2 - \mu^2}
\quad \text{for every} \quad
\mu \in \ooi{0, r}
\]
is satisfied.

This idea was used by Bellazzini \& Siciliano in \cite{bellazziniScalingPropertiesFunctionals2011} to introduce a framework to rule out dichotomy of minimizing sequences of mass-constrained functionals which are invariant by translation in abstract Hilbert subspaces of $L^2 \parens{\real^N}$ (see also \cite{bellazziniStableStandingWaves2011}).

Even though we are interested in minimization problems which are not invariant by translation, we can still consider the strict sub-additivity condition in the present context and adapt Bellazzini \& Siciliano's \cite[Lemma 2.1]{bellazziniScalingPropertiesFunctionals2011} as follows, which marks the beginning of our strategy to avoid dichotomy.

\begin{lem}
\label{dich:lem}
Suppose that $r \in \ooi{0, \infty}$ is such that
\begin{equation}
\label{eqn:SSC}
\I_\alpha \parens{r^2}
<
\I_\alpha \parens{\mu^2} + \I_\alpha \parens{r^2 - \mu^2}
\quad \text{for every} \quad
\mu \in \ooi{0, r}.
\end{equation}
Suppose further that there exist $\set{u_n}_{n \in \nat} \subset W^{1, 2} \parens{r^2}$ and $u_\infty \in W^{1, 2}_\alpha \setminus \set{0}$ such that
\begin{itemize}
\item
$\parens{u_n}_{n \in \nat}$ is a Palais--Smale sequence of
$I_\alpha|_{W^{1, 2}_\alpha \parens{r^2}}$,
\item
$\parens{u_n}_{n \in \nat}$ is a minimizing sequence of
$I_\alpha|_{W^{1, 2}_\alpha \parens{r^2}}$,
\item
$u_n \rightharpoonup u_\infty$ in $W^{1, 2}_\alpha$ as $n \to \infty$
\end{itemize}
and the following conditions are satisfied:
\begin{equation}
\label{eqn:dichotomy:T_1}
T_\alpha \parens{u_\infty}
\leq
\liminf_{n \to \infty}
\parens*{T_\alpha \parens{u_n} - T_\alpha \parens{u_n - u_\infty}},
\end{equation}
\begin{equation}
\label{eqn:dichotomy:T_2}
T_\alpha \parens*{\delta_n \parens{u_n - u_\infty}}
-
T_\alpha \parens{u_n - u_\infty}
\xrightarrow[n \to \infty]{}
0
\quad
\parens*{
\delta_n
:=
\frac{\sqrt{r^2 - \norm{u_\infty}_{L^2}^2}}{\norm{u_n - u_\infty}_{L^2}}
}
\end{equation}
and
\begin{equation}
\label{eqn:dichotomy:T_3}
\liminf_{n, m \to \infty}
\parens*{T_\alpha' \parens{u_n} - T_\alpha' \parens{u_m}} \brackets{u_n - u_m}
\geq
0.
\end{equation}
Then, up to subsequence,
$\norm{u_n - u_\infty}_{W^{1, 2}_\alpha} \to 0$ as $n \to \infty$. In particular, $u_\infty$ solves the abstract minimization problem
\begin{equation}
\label{eqn:AMP:3}
\begin{cases}
I_\alpha \parens{u}
=
\I_\alpha \parens{r^2};
\\
u \in W^{1, 2}_\alpha \parens{r^2},
\end{cases}
\end{equation}
(that is, \eqref{eqn:AMP} in the case $\rho = r$).
\end{lem}

A few remarks are in order.

\begin{rmk}
\label{intro:dich:rmk}
\begin{enumerate}
\item
We used $r$ in the statement of the lemma instead of $\rho$ for notational reasons that will become clear in the statement of Theorem \ref{dich:thm}.
\item
Conditions \eqref{eqn:dichotomy:T_1}--\eqref{eqn:dichotomy:T_3} depend on the set of functions $\set{u_n}_{n \in \nat} \cup \set{u_\infty} \subset W^{1, 2}_\alpha$. These conditions do not depend uniquely on the nonlinear functional
$T_\alpha \colon W^{1, 2}_\alpha \to \real$.
\item
The proof of the lemma uses the limit
\begin{equation}
\label{eqn:H_alpha-convergence}
H_\alpha \parens{u_n - u_m} \xrightarrow[n, m \to \infty]{} 0
\end{equation}
to deduce that $\norm{u_n - u_\infty}_{W^{1, 2}_\alpha} \to 0$ as $n \to \infty$, and the hypothesis $\alpha \in \coi{0, \infty}$ is essential to prove that \eqref{eqn:H_alpha-convergence} holds (see the proof of \eqref{dich:aux:-2}).
\item
For further comments on this result, see Remark \ref{dich:rmk}.
\end{enumerate}
\end{rmk}

It may be technically difficult to prove that \eqref{eqn:SSC} is satisfied. For this reason, it is useful to state an alternative result that uses Bellazzini \& Siciliano's technique of scaling paths introduced in \cite{bellazziniScalingPropertiesFunctionals2011}. We begin by defining what we mean by a scaling path.
\begin{defn}
Consider a $u \in W^{1, 2}_\alpha \setminus \set{0}$. The mapping
$g_u \colon \ooi{0, \infty} \to W^{1, 2}_\alpha$
is said to be an \emph{admissible scaling path (ASP)} of $u$ when
\begin{enumerate}
\item
$g_u$ is continuous;
\item
$g_u \parens{1} = u$;
\item
the functions
$\Theta_{g_u} \colon \ooi{0, \infty} \to \coi{0, \infty}$
and
$h_{\alpha, g_u} \colon \ooi{0, \infty} \to \real$
respectively defined as
\[
\Theta_{g_u} \parens{\theta}
=
\frac{\norm{g_u \parens{\theta}}_{L^2}^2}{\norm{u}_{L^2}^2}
\]
and
\[
h_{\alpha, g_u} \parens{\theta}
=
I_\alpha \parens*{g_u \parens{\theta}}
-
\Theta_{g_u} \parens{\theta} I_\alpha \parens{u}
\]
are differentiable;
\item
$\Theta_{g_u}' \parens{1} \neq 0$.
\end{enumerate}
\end{defn}

With this notion in mind, we developed the following result based on \cite[Theorem 2.1]{bellazziniScalingPropertiesFunctionals2011}.

\begin{thm}
\label{dich:thm}
Suppose that
\begin{enumerate}
\item
\label{dich:cond:1}
$\displaystyle \frac{\I_\alpha \parens{r^2}}{r^2} \xrightarrow[r \to 0^+]{} 0$.
\end{enumerate}
Suppose further that $\rho \in \ooi{0, \infty}$ is such that the following conditions are satisfied.
\begin{enumerate}
\setcounter{enumi}{1}
\item
\label{dich:cond:2}
$- \infty < \I_\alpha \parens{r^2} < 0$ for every $r \in \oci{0, \rho}$.
\item
\label{dich:cond:3}
The function
$\oci{0, \rho} \ni r \mapsto \I_\alpha \parens{r^2} \in \ooi{- \infty, 0}$
is continuous.
\item
\label{dich:cond:4}
For every $r \in \oci{0, \rho}$, there exist
$\set{u_n}_{n \in \nat} \subset W^{1, 2}_\alpha \parens{r^2}$
and $u_\infty \in W^{1, 2}_\alpha \setminus \set{0}$ such that
\begin{itemize}
\item
$\parens{u_n}_{n \in \nat}$ is a Palais--Smale sequence of
$I_\alpha|_{W^{1, 2}_\alpha \parens{r^2}}$,
\item
$\parens{u_n}_{n \in \nat}$ is a minimizing sequence of $I_\alpha|_{W^{1, 2} \parens{r^2}}$,
\item
\eqref{eqn:dichotomy:T_1}--\eqref{eqn:dichotomy:T_3} are satisfied,
\item
$u_n \rightharpoonup u_\infty$ in $W^{1, 2}_\alpha$ as $n \to \infty$.
\end{itemize}
\item
\label{dich:cond:5}
Every
\[
u \in M \parens{\rho^2}
:=
\bigcup_{0 < r \leq \rho}
\set*{
	v \in W^{1, 2}_\alpha \parens{r^2} :
	I_\alpha \parens{v} = \I_\alpha \parens{r^2}
}
\]
has an ASP $g_u \colon \ooi{0, \infty} \to W^{1, 2}_\alpha$ such that
$h_{\alpha, g_u}' \parens{1} \neq 0$.
\end{enumerate}
Then \eqref{eqn:AMP} has a solution.
\end{thm}

Let us briefly comment on the strategy for the proof of the theorem. Lemma \ref{dich:lem:3} will show that due to Conditions \ref{dich:cond:1}--\ref{dich:cond:3}, there exists
$r \in \oci{0, \rho}$ such that \eqref{eqn:SSC} is satisfied. In view of Condition \ref{dich:cond:4}, we can apply Lemma \ref{dich:lem} to deduce that \eqref{eqn:AMP:3} has a solution. Finally, we can use Condition \ref{dich:cond:5} to prove that $r = \rho$.

Even though the conditions in Theorem \ref{dich:thm} may seem more complicated than those in \cite[Theorem 2.1]{bellazziniScalingPropertiesFunctionals2011}, we claim that the proof of \cite[Theorem 2.1]{bellazziniScalingPropertiesFunctionals2011} implicitly uses similar conditions. Indeed, \cite[Theorem 2.1]{bellazziniScalingPropertiesFunctionals2011} mentions \cite[(2.1)]{bellazziniScalingPropertiesFunctionals2011} without previously considering a minimizing sequence denoted by $\parens{u_n}_{n \in \nat}$ and a careful inspection of \cite[Section 3]{bellazziniScalingPropertiesFunctionals2011} shows that the unstated fact that \cite[(2.1)]{bellazziniScalingPropertiesFunctionals2011} would be satisfied for any minimizing sequence with any $L^2$-norm in $\oci{0, \rho}$ is being used to apply \cite[Lemma 2.1]{bellazziniScalingPropertiesFunctionals2011}: see the first paragraph of the Proof of Theorem 2.1 in \cite[p. 2494]{bellazziniScalingPropertiesFunctionals2011}.

These results are proved in Section \ref{dichotomy}.

\subsection{Applications}
\label{intro:applications}
\subsubsection{A Kirchhoff-type equation involving a point interaction}
\label{intro:Kirchhoff}
The \emph{Kirchhoff equation} is a nonlinear wave equation proposed in \cite[Neunundzwanzigste Vorlesung, \S 7]{kirchhoffVorlesungenUeberMathematische1876} as a model for the vibration of elastic strings which takes into account the change of their length produced by transverse vibrations. When searching for standing waves of a generalization of the Kirchhoff equation to three dimensions, we obtain the following \emph{Kirchhoff-type equation} in $\real^3$:
\begin{equation}
\label{Kirchhoff:eqn:no_point_interaction}
-\parens*{1 + \int \abs*{\nabla u \parens{x}}^2 \dif x} \Delta u
+
\omega u
=
u \abs{u}^{p - 2},
\end{equation}
where $2 < p < 6$ and we want to solve for $\omega \in \real$,
$u \colon \real^3 \to \complex$. Problem \eqref{Kirchhoff:eqn:no_point_interaction} and its variations have been already intensely studied in the literature: see \cite{alvesNonlinearPerturbationsPeriodic2012, azzolliniEllipticKirchhoffEquation2012, figueiredoExistenceLeastEnergy2015, heExistenceConcentrationBehavior2012, yeSharpExistenceConstrained2015}, for instance.

The nonlinear functional $H_\alpha \colon W^{1, 2}_\alpha \to \coi{0, \infty}$ extends
\[
W^{1, 2} \ni u
\mapsto 
\norm{u}_{\dot{W}^{1, 2}}^2
=
\int \abs*{\nabla u \parens{x}}^2 \dif x
\in
\coi{0, \infty},
\]
so we expect a \emph{Kirchhoff-type equation with a point interaction} in
$\real^3$ to take the following form:
\begin{equation}
\label{eqn:delta-Kirchhoff}
-\parens*{1 + H_\alpha \parens{u}} \Delta_\alpha u + \omega u
=
u \abs{u}^{p - 2}.
\end{equation}
Due to the Sobolev embeddings of $W^{1, 2}_\alpha$, we can proceed variationally whenever $2 < p < 3$. In this case, we associate \eqref{eqn:delta-Kirchhoff} with the \emph{energy functional} $I_\alpha \colon W^{1, 2}_\alpha \to \real$ defined as
\[
I_\alpha \parens{u}
=
\frac{1}{2} H_\alpha \parens{u}
+
\underbrace{
	\frac{1}{4} B_\alpha \parens{u}
	-
	\frac{1}{p} C \parens{u}
}_{=: T_\alpha \parens{u}},
\]
where $B_\alpha \colon W^{1, 2}_\alpha \to \coi{0, \infty}$ and $C \colon L^p \to \real$ are respectively defined as
\[
B_\alpha \parens{u} = H_\alpha \parens{u}^2
\quad \text{and} \quad
C \parens{u} = \norm{u}_{L^p}^p.
\]
In this context, we understand a \emph{ground state} of \eqref{eqn:delta-Kirchhoff} with \emph{mass} $\rho$ to be a solution to \eqref{eqn:AMP}. We proceed to a rigorous statement of our result about existence of ground states of \eqref{eqn:delta-Kirchhoff}.

\begin{thm}
\label{thm:Kirchhoff}
Suppose that $\alpha \in \coi{0, \infty}$ and $2 < p < \frac{5}{2}$. There exists $\rho_0 \in \ooi{0, \infty}$ such that if $0 < \rho < \rho_0$, then
\begin{enumerate}
\item
\eqref{eqn:delta-Kirchhoff} has a ground state with mass
$\rho$ and
\item
every ground state of \eqref{eqn:delta-Kirchhoff} with mass $\rho$ is in
$W^{1, 2}_\alpha \parens{\rho^2} \setminus W^{1, 2} \parens{\rho^2}$.
\end{enumerate}
\end{thm}

We prove Theorem \ref{thm:Kirchhoff} in Section \ref{Kirchhoff}. In fact, the result follows directly from an application of Lemma \ref{lem:vanishing} and Theorem \ref{dich:thm}, so its proof consists of a verification that their hypotheses are satisfied. A technical limitation of our method is that we can only prove the existence of ground states at sufficiently small masses, which contrasts with the fact that \eqref{Kirchhoff:eqn:no_point_interaction} has ground states with any mass when $2 < p < \frac{5}{2}$ (see \cite[Theorem 1.1]{yeSharpExistenceConstrained2015}). This happens because we can only prove that Condition \ref{dich:cond:5} in Theorem \ref{dich:thm} is satisfied when $\rho$ is a sufficiently small positive number (see the proof of Lemma \ref{Kirchhoff:dich:lem:cond:5}).

\subsubsection{The Schrödinger--Poisson system involving a point interaction}

Consider the following \emph{nonlinear Schrödinger--Poisson system} in $\real^3$:
\[
\begin{cases}
- \Delta u + \omega u + \psi u = u \abs{u}^{p - 2};
\\
- \Delta \psi = 4 \pi \abs{u}^2,
\end{cases}
\]
where $\omega \in \real$, $2 < p < 6$ and the unknowns are $u$,
$\psi \colon \real^3 \to \complex$. Solutions to this system describe the spatial profile of standing waves of the NLSE that suffer electrostatic self-interaction through the minimal coupling rule as modeled by Maxwell's electromagnetism (for details about the physical deduction of this system, see \cite[Section 2]{benciEigenvalueProblemSchrodingerMaxwell1998}). This system has been already widely studied in the literature: see \cite{ambrosettiMULTIPLEBOUNDSTATES2008, azzolliniGroundStateSolutions2008, bellazziniScalingPropertiesFunctionals2011, bellazziniStableStandingWaves2011, benciEigenvalueProblemSchrodingerMaxwell1998, daprileSolitaryWavesNonlinear2004, ruizSchrodingerPoissonEquation2006}, for instance.

Inspired by this system, we propose the problem of existence of ground states of the \emph{nonlinear Schrödinger--Poisson system with a point interaction} in $\real^3$ that follows:
\begin{equation}
\label{eqn:delta-SP}
\begin{cases}
\op u + \omega u + \psi u = u \abs{u}^{p - 2};
\\
- \Delta \psi = 4 \pi \abs{u}^2,
\end{cases}
\end{equation}
where $\alpha \in \coi{0, \infty}$, $2 < p < 3$ and we want to solve for
$\omega \in \real$, $u \colon \real^3 \to \complex$,
$\psi \colon \real^3 \to \coi{0, \infty}$.

At this point, it is worth recalling the following corollary of the Hardy--Littlewood--Sobolev Inequality (see \cite[Theorem 4.3]{liebAnalysis2001}).

\begin{cor}
\label{cor:HLS}
Suppose that $p \in \ooi{1, \frac{3}{2}}$ and $r \in \ooi{3, \infty}$ are such that $\frac{1}{p} = \frac{2}{3} + \frac{1}{r}$. Then
$\psi_u := \abs{\cdot}^{- 1} \ast u \in L^r$
for every $u \in L^p$ and the linear mapping
\[L^p \ni u \mapsto \psi_u \in L^r\]
is continuous.
\end{cor}

Consider a $u \in W^{1, 2}_\alpha$. It follows from the corollary that
$\abs{u}^2 \in L^p$ for every $p \in \ooi{1, \frac{3}{2}}$, so $\psi_{\abs{u}^2} \in L^r$ for every $r \in \ooi{3, \infty}$. An application of Hölder's Inequality shows that
$\psi_{\abs{u}^2} u \in L^2$ and $\psi_{\abs{u}^2} \abs{u}^2 \in L^1$. As such, we will actually consider the following reduction of \eqref{eqn:delta-SP}:
\begin{equation}
\label{eqn:reduced-delta-SP}
\op u + \omega u + \psi_{\abs{u}^2} u = u \abs{u}^{p - 1},
\end{equation}
where $2 < p < 3$. This reduction is usually employed in papers concerned with the Schrödinger--Poisson system, and is motivated by the fact that $\abs{\cdot}^{- 1}$ is the fundamental solution of $- \frac{1}{4 \pi} \Delta$.

We associate \eqref{eqn:reduced-delta-SP} with the energy functional
$I_\alpha \colon W^{1, 2}_\alpha \to \real$
defined as
\[
I_\alpha \parens{u}
=
\frac{1}{2} H_\alpha \parens{u}
+
\underbrace{
	\frac{1}{4} B \parens{u}
	-
	\frac{1}{p} C \parens{u}
}_{=: T_\alpha \parens{u}},
\]
where $B \colon L^{12 / 5} \to \coi{0, \infty}$ and $C \colon L^p \to \coi{0, \infty}$ are respectively defined as
\[
B \parens{u} = \int \parens*{\psi_{\abs{u}^2} \abs{u}^2}
\quad \text{and} \quad
C \parens{u} = \norm{u}_{L^p}^p.
\]
Let us state our result about existence of ground states of \eqref{eqn:reduced-delta-SP}.

\begin{thm}
\label{thm:delta-SP}
Suppose that $\alpha \in \coi{0, \infty}$ and $2 < p < \frac{5}{2}$. There exists $\rho_0 \in \ooi{0, \infty}$ such that if $0 < \rho < \rho_0$, then
\begin{enumerate}
\item
\eqref{eqn:reduced-delta-SP} has a ground state with mass
$\rho$ and
\item
every ground state of \eqref{eqn:reduced-delta-SP} with mass $\rho$ is in
$W^{1, 2}_\alpha \parens{\rho^2} \setminus W^{1, 2} \parens{\rho^2}$.
\end{enumerate}
\end{thm}

Like Theorem \ref{thm:Kirchhoff}, the proof of Theorem \ref{thm:delta-SP} consists of a collection of verifications that the hypotheses of Lemma \ref{lem:vanishing} and Theorem \ref{dich:thm} are satisfied. These are done in Section \ref{SP}.

\begin{rmk}
This proof technique may also be applied to prove the existence of ground states at sufficiently small masses for a similar system under the additional effect of a point interaction. That is, consider the following \emph{nonlinear Schrödinger--Bopp--Podolsky system} in $\real^3$ introduced by d'Avenia \& Siciliano in \cite{daveniaNonlinearSchrodingerEquation2019}:
\[
\begin{cases}
	-\Delta u + \omega u + \psi u = u \abs{u}^{p - 2};\\
	\beta^2 \Delta^2 \psi - \Delta \psi = 4 \pi \abs{u}^2,
\end{cases}
\]
where $0 < \beta < \infty$, $2 < p < 6$ and we want to solve for
$\omega \in \real$, $u \colon \real^3 \to \complex$,
$\psi \colon \real^3 \to \coi{0, \infty}$. This system is obtained when considering the Bopp--Podolsky electromagnetism instead of Maxwell's. A motivation for this alternate theory is that it solves the \emph{infinity problem}, i.e., the fact that the electrical field induced by a point charge is not in $\dot{W}^{1, 2}$. We refer the reader to \cite{damianCriticalSchrodingerBopp2024, depaularamosConcentratedSolutionsSchrodinger2024, depaularamosExistenceLimitBehavior2023, figueiredoMultipleSolutionsSchrodinger2023, mascaroPositiveSolutionsSchrodingerBoppPodolsky2022} for a few studies of this system.
\end{rmk}

\subsection*{Acknowledgements}

The author is grateful for the anonymous reviewer of a previous version of this text, whose suggestions greatly helped to improve the paper. This study was financed in part by the Coordenação de Aperfeiçoamento de Pessoal de Nível Superior - Brasil (CAPES) - Finance Code 001.

\section{Technical preliminaries}
\label{prelim}

\subsection{Properties of the Green's functions}

In fact, $G_\lambda - G_\mu$ is more regular than $G_\lambda$, $G_\mu$.

\begin{lem}
\label{lem:G_lambda-G_mu-in-W^{2, 2}}
It holds that $G_\lambda - G_\mu \in W^{2, 2}$ for every $\lambda, \mu \in \ooi{0, \infty}$.
\end{lem}
\begin{proof}
It is easy to verify explicitly that
\begin{equation}
\label{eqn:G_lambda-G_mu-in-W^{1, 2}}
G_\lambda - G_\mu \in W^{1, 2}.
\end{equation}

Let us prove that
\begin{equation}
\label{eqn:G_lambda-G_mu-in-L^2}
\abs*{\parens{G_\lambda - G_\mu}''} \in L^2.
\end{equation}
In view of the fact that
\[- \Delta G_\lambda + \lambda G_\lambda = - \Delta G_\mu + \mu G_\mu = \delta_0\]
in the sense of distributions, we deduce that
\begin{equation}
\label{eqn:G_lambda-G_mu}
- \Delta \parens{G_\lambda - G_\mu} = \mu G_\mu - \lambda G_\lambda
\quad \text{in} \quad
\real^3 \setminus \set{0}.
\end{equation}
As $G_\lambda, G_\mu \in L^2$, it follows that
\begin{equation}
\label{eqn:lapl-in-L^2}
- \Delta \parens{G_\lambda - G_\mu} \in L^2.
\end{equation}
Due to the Chain Rule and the Divergence Theorem,
\begin{equation}
\label{eqn:Dirichlet-energy}
\int \abs*{\Delta \phi \parens{x}}^2 \dif x
=
\int \abs*{\phi'' \parens{x}}^2 \dif x
\quad \text{for every} \quad
\phi \in C_c^\infty.
\end{equation}
It follows from \eqref{eqn:lapl-in-L^2} and \eqref{eqn:Dirichlet-energy}that
$\abs{\parens{G_\lambda - G_\mu}''} \in L^2$.

Finally, the inclusion $G_\lambda - G_\mu \in W^{2, 2}$ follows from \eqref{eqn:G_lambda-G_mu-in-W^{1, 2}} and \eqref{eqn:G_lambda-G_mu-in-L^2}.
\end{proof}

Due to Lemma \ref{lem:G_lambda-G_mu-in-W^{2, 2}}, it makes sense to consider the continuous linear functional
\[
W^{1, 2} \ni u \mapsto \angles{G_\lambda - G_\mu \mid u}_{\dot{W}^{1, 2}} \in \complex.
\]
In this context, the following result is a corollary of \eqref{eqn:G_lambda-G_mu}.

\begin{lem}
\label{lem:G_lambda-G_mu}
Suppose that $\lambda, \mu \in \ooi{0, \infty}$. Then
\[
\int \nabla \parens{G_\lambda - G_\mu} \parens{x} \cdot \nabla \phi \parens{x} \dif x
=
\int \parens*{\mu G_\mu \parens{x} - \lambda G_\lambda \parens{x}} \phi \parens{x} \dif x
\]
for every $\phi \in W^{1, 2}$.
\end{lem}

It follows from \eqref{eqn:Green's-fct} that
\[
\int
	\phi \parens{x}
	\parens*{- \Delta G_\lambda \parens{x} + \lambda G_\lambda \parens{x}} 
\dif x
=
\phi \parens{0}
\]
for every $\phi \in W^{2, 2}$. In fact, a similar identity holds when
$\parens{- \Delta + \lambda}$ acts on $\phi$ instead of on $G_\lambda$.

\begin{lem}
\label{lem:symmetric}
It holds that
\[
\angles{- \Delta \phi + \lambda \phi \mid G_\lambda}_{L^2}
=
\phi \parens{0}
\]
for every $\lambda \in \ooi{0, \infty}$ and $\phi \in W^{2, 2}$.
\end{lem}
\begin{proof}
Consider an $\eps \in \ooi{0, \infty}$. Due to the Chain Rule,
\[
- G_\lambda \Delta \phi
=
- \phi \Delta G_\lambda
+
\mathrm{div} \parens{\phi \nabla G_\lambda - G_\lambda \nabla \phi}
\]
a.e. in $\real^3$. It follows from the Divergence Theorem \cite[Theorem 6.3.5]{willemFunctionalAnalysisFundamentals2022} that
\begin{multline*}
-
\int_{\real^3 \setminus B \parens{0, \eps}}
	G_\lambda \parens{x} \Delta \phi \parens{x}
\dif x
=
-
\int_{\real^3 \setminus B \parens{0, \eps}}
	\phi \parens{x} \Delta G_\lambda \parens{x}
\dif x
+
\\
+
\frac{e^{- \sqrt{\lambda} \eps}}{4 \pi \eps}
\int_{\eps \mathbb{S}^2}
	\frac{\nabla \phi \parens{x} \cdot x}{\abs{x}}
\dif \sigma_\eps \parens{x}
+
\frac{\parens{\sqrt{\lambda} \eps + 1} e^{- \sqrt{\lambda} \eps}}{4 \pi \eps^2}
\int_{\eps \mathbb{S}^2}
	\phi \parens{x}
\dif \sigma_\eps \parens{x}
\end{multline*}
where $\sigma_\eps$ denotes the surface measure of $\mathbb{S}^2$. A change of variable shows that
\[
\frac{e^{- \sqrt{\lambda} \eps}}{4 \pi \eps}
\int_{\eps \mathbb{S}^2}
	\frac{\nabla \phi \parens{x} \cdot x}{\abs{x}}
\dif \sigma_\eps \parens{x}
+
\frac{\parens{\sqrt{\lambda} \eps + 1} e^{- \sqrt{\lambda \eps}}}{4 \pi \eps^2}
\int_{\eps \mathbb{S}^2}
	\phi \parens{x}
\dif \sigma_\eps \parens{x}
\xrightarrow[\eps \to 0^+]{}
0.
\]
The result follows from \eqref{eqn:Green's-fct} because
\begin{align*}
\angles{- \Delta \phi + \lambda \phi \mid G_\lambda}_{L^2}
&=
\lim_{\eps \to 0^+}
\int_{\real^3 \setminus B \parens{0, \eps}}
	G_\lambda \parens{x} \parens*{- \Delta \phi \parens{x} + \lambda \phi \parens{x}} 
\dif x;
\\
&=
\lim_{\eps \to 0^+}
\int_{\real^3 \setminus B \parens{0, \eps}}
	\phi \parens{x}
	\parens*{- \Delta G_\lambda \parens{x} + \lambda G_\lambda \parens{x}} 
\dif x;
\\
&=
\int
	\phi \parens{x}
	\parens*{- \Delta G_\lambda \parens{x} + \lambda G_\lambda \parens{x}} 
\dif x;
\\
&=
\phi \parens{0}.
\end{align*}
\end{proof}

\subsection{Properties of $S_\alpha$, $H_\alpha$ and $W^{1, 2}_\alpha$; a useful inequality}
\label{energy-space}

The following result shows that the definition on Section \ref{intro:energy-space} gives a well-defined sesquilinear form
$S_\alpha \colon \Dom \brackets{\op} \times \Dom \brackets{\op} \to \complex$.

\begin{lem}
\label{lem:S_alpha-is-well-defined}
The sesquilinear form $S_\alpha$ is well-defined.
\end{lem}
\begin{proof}
We want to prove that
\begin{multline}
\label{prelim:aux:0}
\angles{\phi_\lambda \mid \psi_\lambda}_{\dot{W}^{1, 2}}
-
\lambda
\angles{\overline{q_v} \phi_\lambda + q_u \overline{\psi_\lambda} \mid G_\lambda}_{L^2}
+
\frac{\sqrt{\lambda}}{8 \pi} q_u \overline{q_v}
=
\\
=
\angles{\phi_\mu \mid \psi_\mu}_{\dot{W}^{1, 2}}
-
\mu
\angles{\overline{q_v} \phi_\mu + q_u \overline{\psi_\mu} \mid G_\mu}_{L^2}
+
\frac{\sqrt{\mu}}{8 \pi} q_u \overline{q_v},
\end{multline}
where
\[u = \phi_\lambda + q_u G_\lambda = \phi_\mu + q_u G_\mu \in \Dom \brackets{\op}\]
and
\[v = \psi_\lambda + q_v G_\lambda = \psi_\mu + q_v G_\mu \in \Dom \brackets{\op}.\]
In view of the equality
$\phi_\lambda - \phi_\mu = q_u \parens{G_\mu - G_\lambda}$,
Lemma \ref{lem:G_lambda-G_mu} shows that
\[
\angles{\phi_\lambda \mid \psi_\lambda}_{\dot{W}^{1, 2}}
-
\lambda
\angles{q_u G_\lambda \mid \psi_\lambda}_{L^2}
=
-
\mu
\angles{q_u G_\mu \mid \psi_\lambda}_{L^2}
+
\angles{\phi_\mu \mid \psi_\lambda}_{\dot{W}^{1, 2}}.
\]
Similarly,
\[
\angles{G_\mu \mid \psi_\lambda}_{L^2}
=
\frac{q_v}{8 \pi \sqrt{\mu}}
-
\frac{q_v}{4 \pi \parens{\sqrt{\lambda} + \sqrt{\mu}}}
+
\angles{G_\mu \mid \psi_\mu}_{L^2}.
\]
Therefore,
\begin{multline}
\label{prelim:aux:1}
\angles{\phi_\lambda \mid \psi_\lambda}_{\dot{W}^{1, 2}}
-
\lambda
\angles{q_u G_\lambda \mid \psi_\lambda}_{L^2}
=
\\
=
\overline{q_u} q_v
\parens*{
	\frac{\mu}{4 \pi \parens*{\sqrt{\lambda} + \sqrt{\mu}}}
	-
	\frac{\sqrt{\mu}}{8 \pi}
}
-
\mu
\angles{q_u G_\mu \mid \psi_\mu}_{L^2}
+
\angles{\phi_\mu \mid \psi_\lambda}_{\dot{W}^{1, 2}}.
\end{multline}
An analogous argument shows that
\begin{multline}
\label{prelim:aux:2}
\angles{\phi_\mu \mid \psi_\mu}_{\dot{W}^{1, 2}}
-
\mu
\angles{\phi_\mu \mid q_v G_\mu}_{L^2}
=
\\
=
\overline{q_u} q_v
\parens*{
	\frac{\lambda}{4 \pi \parens*{\sqrt{\lambda} + \sqrt{\mu}}}
	-
	\frac{\sqrt{\lambda}}{8 \pi}
}
-
\lambda
\angles{\phi_\lambda \mid q_v G_\lambda}_{L^2}
+
\angles{\phi_\mu \mid \psi_\lambda}_{\dot{W}^{1, 2}}.
\end{multline}
Finally, \eqref{prelim:aux:0} follows directly from \eqref{prelim:aux:1} and \eqref{prelim:aux:2}.
\end{proof}

The next result follows directly from the definition of $H_\alpha$.

\begin{lem}
\label{lem:properties-of-H_alpha}
\begin{enumerate}
\item
Given $\delta \in \complex$ and $u \in W^{1, 2}_\alpha$,
$H_\alpha \parens{\delta u} = \abs{\delta}^2 H_\alpha \parens{u}$.
\item
If
$u_n \rightharpoonup u_\infty$ in $W^{1, 2}_\alpha$ as $n \to \infty$, then
\[
H_\alpha \parens{u_n - u_\infty}
-
H_\alpha \parens{u_n}
+
H_\alpha \parens{u_\infty}
\xrightarrow[n \to \infty]{}
0.
\]
\end{enumerate}
\end{lem}

The following result about weak convergence in $W^{1, 2}_\alpha$ is a corollary of the fact that given a fixed $\lambda \in \ooi{0, \infty}$,
\[
W^{1, 2}_\alpha \ni \phi_\lambda + q G_\lambda \mapsto q \in \complex
\quad \text{and} \quad
W^{1, 2}_\alpha \ni \phi_\lambda + q G_\lambda \mapsto \phi_\lambda \in W^{1, 2}
\]
are continuous linear mappings.
\begin{lem}
\label{lem:weak-convergence}
Suppose that
$u_\infty = \phi_{\lambda, \infty} + q_\infty G_\lambda \in W^{1, 2}_\alpha$
and
\[\set{u_n = \phi_{\lambda, n} + q_n G_\lambda}_{n \in \nat} \subset W^{1, 2}_\alpha.\]
The following equivalence holds:
$u_n \rightharpoonup u_\infty$ in $W^{1, 2}_\alpha$ as $n \to \infty$
if, and only if,
\[
\phi_{\lambda, n} \xrightharpoonup[n \to \infty]{W^{1, 2}} \phi_{\lambda, \infty}
\quad \text{and} \quad
q_n \xrightarrow[n \to \infty]{} q_\infty.
\]
\end{lem}

We also need the following generalization of the Gagliardo--Nirenberg inequality.

\begin{prop}[{\cite[Proposition II.1]{adamiExistenceStructureRobustness2022}}]
\label{prelim:prop:GN}
Suppose that $2 < r < 3$. There exists $K_r > 0$ such that
\[
\norm{u}_{L^r}^r
\leq
K_r
\parens*{
	\norm{\phi}_{\dot{W}^{1, 2}}^{\frac{3 \parens{r - 2}}{2}}
	\norm{\phi}_{L^2}^{\frac{6 - r}{2}}
	+
	\frac{\abs{q}^r}{\lambda^{\frac{3 - r}{2}}}
}
\]
for every $u = \phi + q G_\lambda \in W^{1, 2}_\alpha$. We can also associate each
$\eps \in \ooi{0, \infty}$ with a $K_{r, \eps} \in \ooi{0, \infty}$ such that
\[
\norm{u}_{L^r}^r
\leq
K_{r, \eps}
\parens*{
	\norm{\phi}_{\dot{W}^{1, 2}}^{\frac{3 \parens{r - 2}}{2}}
	\norm{u}_{L^2}^{\frac{6 - r}{2}}
	+
	\abs{q}^{3 \parens{r - 2}}
	\norm{u}_{L^2}^{2 \parens{3 - r}}
}
\]
for every
\[
u
=
\phi + q G_{\eps \abs{q}^4 / \norm{u}_{L^2}^4}
\in W^{1, 2}_\alpha \setminus W^{1, 2}.
\]
\end{prop}

\section{Proofs of results in Section \ref{intro:existence-of-soln}}
\label{abstract}

\subsection{Avoiding vanishing}
\label{vanishing}

\begin{proof}[Proof of Lemma \ref{lem:vanishing}]
\emph{1. Setup.}
By contradiction, suppose that there exists
\[
\set{u_n = \phi_{\lambda, n} + q_n G_\lambda}_{n \in \nat}
\subset
W^{1, 2}_\alpha \parens{\rho^2}
\]
such that $I_\alpha \parens{u_n} \to \I_\alpha \parens{\rho^2}$ and
$q_n \to 0$ as $n \to \infty$. Due to the limit $q_n \to 0$ as $n \to \infty$, it follows from \eqref{cond:vanishing:2} that
\begin{equation}
\label{vanishing:aux:10}
\parens{\norm{\phi_{\lambda, n}}_{W^{1, 2}}}_{n \in \nat}
~ \text{is bounded}.
\end{equation}
\\ \\
\noindent \emph{2. A preliminary result.}
We want to prove that
\begin{equation}
\label{vanishing:aux:11}
I \parens{\xi_n} - I_\alpha \parens{u_n}
\xrightarrow[n \to \infty]{}
0,
\end{equation}
where
\[
\xi_n
=
\frac{\rho}{\norm{\phi_{\lambda, n}}_{L^2}} \phi_{\lambda, n}
\in
W^{1, 2} \parens{\rho^2} \subset W^{1, 2}_\alpha \parens{\rho^2}
\]
for every $n \in \nat$. On one hand, it follows from the limit $q_n \to 0$ as
$n \to \infty$ that
\[
\norm{\xi_n}_{\dot{W}^{1, 2}}^2 - H_\alpha \parens{u_n}
\xrightarrow[n \to \infty]{}
0.
\]
On the other hand, it follows from \eqref{intro:eqn:T_1}, \eqref{intro:eqn:T_2} and \eqref{vanishing:aux:10} that
\[
\abs*{T \parens{\xi_n} - T \parens{u_n}}
\leq
\parens*{
	\max_{t \in \cci{0, 1}}
	\norm*{T' \parens{\xi_n + t u_n}}_{\mathcal{L} \parens{W^{1, 2}_\alpha, \real}}
}
\norm{u_n - \xi_n}_{W^{1, 2}_\alpha}
\xrightarrow[n \to \infty]{}
0.
\]
As such, we deduce that \eqref{vanishing:aux:11} holds.
\\ \\
\noindent \emph{3. The contradiction.}
Due to the limit $I_\alpha \parens{u_n} \to \I_\alpha \parens{\rho^2}$ as $n \to \infty$, it follows from \eqref{vanishing:aux:11} that
\[
\I \parens{\rho^2}
\leq
\lim_{n \to \infty} I \parens{\xi_n}
=
\lim_{n \to \infty} I_\alpha \parens{u_n}
=
\I_\alpha \parens{\rho^2},
\]
which contradicts \eqref{cond:vanishing:1}.
\end{proof}

\begin{proof}[Proof of Lemma \ref{lem:calI_alpha<calI}]
The nonlinear functional $I_\alpha$ extends $I$, so
$\I_\alpha \parens{\rho^2} \leq \I \parens{\rho^2}$. By contradiction, suppose that
$\I_\alpha \parens{\rho^2} = \I \parens{\rho^2}$. The function $v$ solves \eqref{eqn:AMP:2}, so it also solves \eqref{eqn:AMP}. As
$\parens{I_\alpha|_{W^{1, 2} \parens{\rho^2}}}' \parens{v} = 0$, it follows from the Lagrange Multiplier Theorem that there exists $\omega \in \real$ such that
\[
I_\alpha' \parens{v} \brackets{u}
+
\omega \Re \brackets*{\int \overline{v \parens{x}} u \parens{x} \dif x}
=
0
\]
for every $u \in W^{1, 2}_\alpha$. That is,
\begin{equation}
\label{vanishing:aux:0}
\Re \brackets*{
	S_\alpha \parens{v, u}
	+
	\int
		\parens*{\overline{z_v \parens{x}} + \omega \overline{v \parens{x}}}
		u \parens{x}
	\dif x
}
=
0
\quad \text{for every} \quad u \in W^{1, 2}_\alpha.
\end{equation}
Integration against functions in $W^{1, 2}$ shows that
\begin{equation}
\label{vanishing:aux:1}
- \Delta v + z_v + \omega v = 0 \quad \text{in} \quad L^2.
\end{equation}
In view of \eqref{vanishing:aux:1}, it suffices to take $u = G_\lambda$ in \eqref{vanishing:aux:0} to obtain
\[
\angles{- \Delta v + \lambda v \mid G_\lambda}_{L^2}
=
0.
\]
An application of Lemma \ref{lem:symmetric} shows that $v \parens{0} = 0$, which is absurd because we suppose that $v \parens{0} \neq 0$.
\end{proof}

\subsection{Avoiding dichotomy}
\label{dichotomy}

Instead of directly proving Lemma \ref{dich:lem}, we will prove the following more precise result which consists of an adaptation of \cite[Lemma 2.1]{bellazziniScalingPropertiesFunctionals2011}.

\begin{lem}
\label{dich:lem:2}
Suppose that $r \in \ooi{0, \infty}$ is such that \eqref{eqn:SSC} holds. Suppose further that $\parens{u_n}_{n \in \nat}$ is a minimizing sequence of
$I_\alpha|_{W^{1, 2}_\alpha \parens{r^2}}$ such that
$u_n \rightharpoonup u_\infty$ in $W^{1, 2}_\alpha$ as
$n \to \infty$, $\mu := \norm{u_\infty}_{L^2} > 0$, \eqref{eqn:dichotomy:T_1} and \eqref{eqn:dichotomy:T_2} are satisfied. Then $\mu = r$. If we suppose further that
$\parens{u_n}_{n \in \nat}$ is a Palais--Smale sequence of
$I_\alpha|_{W^{1, 2}_\alpha \parens{r^2}}$ and \eqref{eqn:dichotomy:T_3} is satisfied, then $\norm{u_n - u_\infty}_{W^{1, 2}_\alpha} \to 0$ as $n \to \infty$ up to subsequence. In particular, $I_\alpha \parens{u_\infty} = \I_\alpha \parens{r^2}$.
\end{lem}
\begin{proof}
\emph{0. A preliminary.}
We know that $W^{1, 2}_\alpha \hookrightarrow L^2$ and we suppose that
$u_n \rightharpoonup u_\infty$ in $W^{1, 2}_\alpha$ as $n \to \infty$, so
\begin{equation}
\label{dich:aux:0}
u_n \xrightharpoonup[n \to \infty]{L^2} u_\infty.
\end{equation}

\noindent \emph{1. Proof that $\mu = r$.}
The norm $\norm{\cdot}_{L^2}$ is weakly lower semi-continuous and we suppose that
$\mu > 0$, so it follows from \eqref{dich:aux:0} that $0 < \mu \leq r$. By contradiction, suppose that $\mu < r$.
\\ \\
\noindent \emph{1.1. A preliminary result.}
Let us prove that
\begin{equation}
\label{dich:aux:-1}
\I_\alpha \parens{r^2 - \mu^2}
+
\I_\alpha \parens{\mu^2}
\leq
\I_\alpha \parens{r^2} + \eps
\quad \text{for every} \quad
\eps \in \ooi{0, \infty}.
\end{equation}
Consider a fixed $\eps \in \ooi{0, \infty}$. Due to Lemma \ref{lem:properties-of-H_alpha},
\begin{multline}
\label{dich:aux:1}
\I_\alpha \parens{r^2 - \mu^2}
+
\I_\alpha \parens{\mu^2}
\leq
I_\alpha \parens*{\delta_n \parens{u_n - u_\infty}}
+
I_\alpha \parens{u_\infty}
=
\\
=
\frac{1}{2}
\parens*{
	\delta_n^2 H_\alpha \parens{u_n - u_\infty}
	+
	H_\alpha \parens{u_\infty}
	-
	H_\alpha \parens{u_n}
}
+
\\
+
\parens*{
	T_\alpha \parens*{\delta_n \parens{u_n - u_\infty}}
	+
	T_\alpha \parens{u_\infty}
	-
	T_\alpha \parens{u_n}
}
+
I_\alpha \parens{u_n}
~ \text{for every} ~ n \in \nat.
\end{multline}
By hypothesis,
\begin{equation}
\label{dich:aux:1.5}
I_\alpha \parens{u_n} \xrightarrow[n \to \infty]{} \I_\alpha \parens{r^2}.
\end{equation}
Due to \eqref{eqn:dichotomy:T_1} and \eqref{eqn:dichotomy:T_2},
\begin{equation}
\label{dich:aux:2}
T_\alpha \parens*{\delta_n \parens{u_n - u_\infty}}
+
T_\alpha \parens{u_\infty}
-
T_\alpha \parens{u_n}
\leq
\eps
~ \text{if} ~ n \in \nat ~ \text{is sufficiently large}.
\end{equation}
As $u_n \rightharpoonup u_\infty$ in $W^{1, 2}_\alpha$ as $n \to \infty$, it follows that
\begin{equation}
\label{dich:aux:3}
\set{u_n}_{n \in \nat} ~ \text{is a bounded subset of} ~ W^{1, 2}_\alpha.
\end{equation}
It is clear that \eqref{dich:aux:0} implies $\delta_n \to 1$ as $n \to \infty$. In view of \eqref{dich:aux:3} and this limit, we deduce that
\begin{equation}
\label{dich:aux:5}
\parens{\delta_n^2 - 1} H_\alpha \parens{u_n - u_\infty} \xrightarrow[n \to \infty]{} 0.
\end{equation}
Due to \eqref{dich:aux:5} and Lemma \ref{lem:properties-of-H_alpha},
\begin{equation}
\label{dich:aux:6}
\delta_n^2 H_\alpha \parens{u_n - u_\infty}
+
H_\alpha \parens{u_\infty}
-
H_\alpha \parens{u_n}
\xrightarrow[n \to \infty]{}
0.
\end{equation}
Finally, \eqref{dich:aux:-1} follows from \eqref{dich:aux:1}--\eqref{dich:aux:2} and \eqref{dich:aux:6}.
\\ \\
\noindent \emph{1.2. Conclusion.}
The facts \eqref{dich:aux:-1} and \eqref{eqn:SSC} are contradictory, so $\mu = r$.
\\ \\
\noindent \emph{2. Proof of the second conclusion.}
Due to \eqref{dich:aux:0} and the fact that $\mu = r$, it follows that
\begin{equation}
\label{dich:aux:7}
\norm{u_n - u_\infty}_{L^2} \xrightarrow[n \to \infty]{} 0.
\end{equation}
\\ \\
\noindent \emph{2.1. A preliminary result.}
Let us prove that, up to subsequence,
\begin{equation}
\label{dich:aux:-2}
H_\alpha \parens{u_n - u_m} \xrightarrow[n, m \to \infty]{} 0.
\end{equation}
Given $n \in \nat$, let $f_n \in \mathcal{L} \parens{W^{1, 2}_\alpha, \real}$ be defined as
\[
f_n \brackets{v}
=
\Re \brackets*{\int \overline{u_n \parens{x}} v \parens{x} \dif x}.
\]
The set $F_n := \set{t f_n : t \in \real}$ is a finite-dimensional subspace of
$\mathcal{L} \parens{W^{1, 2}_\alpha, \real}$, so it is complemented and there exists a canonical projection
\[\mathrm{proj}_{F_n} \colon \mathcal{L} \parens{W^{1, 2}_\alpha, \real} \to F_n.\]
As such, we let $\nu_n \in \real$ be such that
$\mathrm{proj}_{F_n} I_\alpha' \parens{u_n} = \nu_n f_n$.
Due to the fact that $\parens{u_n}_{n \in \nat}$ is a Palais--Smale sequence of $I_\alpha|_{W^{1, 2}_\alpha \parens{r^2}}$,
\begin{equation}
\label{eqn:dich:aux:7}
\norm{
	I_\alpha' \parens{u_n} - \nu_n f_n
}_{\mathcal{L} \parens{W^{1, 2}_\alpha, \real}}
\xrightarrow[n \to \infty]{}
0.
\end{equation}
As $\parens{u_n}_{n \in \nat}$ is bounded in $W^{1, 2}_\alpha$, we obtain
\begin{equation}
\label{eqn:dich:aux:5}
\parens*{I' \parens{u_n} - \nu_n f_n} \brackets{u_n}
\xrightarrow[n \to \infty]{}
0.
\end{equation}
The set $\set{\nu_n}_{n \in \nat}$ is bounded, so $\parens{\nu_n}_{n \in \nat}$ is convergent up to subsequence. We obtain
\begin{equation}
\label{eqn:dich:aux:4}
\parens{\nu_n - \nu_m} f_m \brackets{u_n - u_m}
\xrightarrow[n, m \to \infty]{}
0.
\end{equation}
Clearly,
\begin{multline*}
\parens*{I_\alpha' \parens{u_n} - \nu_n f_n} \brackets{u_n - u_m}
-
\parens*{I_\alpha' \parens{u_m} - \nu_m f_m} \brackets{u_n - u_m}
=
\\
=
\parens*{I_\alpha' \parens{u_n} - I_\alpha' \parens{u_m}} \brackets{u_n - u_m}
-
\parens*{\nu_n f_n - \nu_m f_m} \brackets{u_n - u_m}
=
\\
=
\frac{1}{2} H_\alpha \parens{u_n - u_m}
+
\parens*{T' \parens{u_n} - T' \parens{u_m}} \brackets{u_n - u_m}
+
\\
+
\parens{\nu_m - \nu_n} f_n \brackets{u_n - u_m}
+
\nu_m \parens{f_m - f_n} \brackets{u_n - u_m}.
\end{multline*}
The $\liminf$ is super-additive, so \eqref{eqn:dich:aux:7} shows that
\begin{multline*}
\frac{1}{2}
\liminf_{n, m \to \infty} \parens*{H_\alpha \parens{u_n - u_m}}
+
\liminf_{n, m \to \infty} \parens*{
	\parens*{T' \parens{u_n} - T' \parens{u_m}} \brackets{u_n - u_m}
}
+
\\
+
\liminf_{n, m \to \infty} \parens*{
	\parens{\nu_m - \nu_n} f_n \brackets{u_n - u_m}
}
+
\liminf_{n, m \to \infty} \parens*{
	\nu_m \parens{f_m - f_n} \brackets{u_n - u_m}
}
\leq
\\
\leq
\liminf_{n, m \to \infty} \left(
	\frac{1}{2} H_\alpha \parens{u_n - u_m}
	+
	\parens*{T' \parens{u_n} - T' \parens{u_m}} \brackets{u_n - u_m}
	+
\right.
	\\
\left.
	+
	\parens{\nu_m - \nu_n} f_n \brackets{u_n - u_m}
	+
	\nu_m \parens{f_m - f_n} \brackets{u_n - u_m}
\right)
=
0.
\end{multline*}
In view of \eqref{eqn:dichotomy:T_3} and \eqref{eqn:dich:aux:4},
$\liminf_{n, m \to \infty} \parens*{H_\alpha \parens{u_n - u_m}} \leq 0$.
The nonlinear functional $H_\alpha$ is non-negative. We conclude that, up to subsequence, \eqref{dich:aux:-2} is satisfied.
\\ \\
\noindent \emph{2.2. Conclusion.}
In view of \eqref{dich:aux:7} and \eqref{dich:aux:-2}, we deduce that
$u_n \to u_\infty$ in $W^{1, 2}_\alpha$ as $n \to \infty$. Finally, the equality $I_\alpha \parens{u_\infty} = \I_\alpha \parens{r^2}$ follows from continuity.
\end{proof}

Due to the similarity between Lemma \ref{dich:lem:2} and \cite[Lemma 2.1]{bellazziniStableStandingWaves2011}, it is worth doing a systematic comparison between both results.
\begin{rmk}
\label{dich:rmk}
\begin{enumerate}
\item
The method of proof is the same.
\item
Lemma \ref{dich:lem:2} is not implied by \cite[Lemma 2.1]{bellazziniStableStandingWaves2011}. Indeed, the energy space
$W^{1, 2}_\alpha$ does not respect the hypotheses in Bellazzini \& Siciliano's result. Moreover, \eqref{eqn:dichotomy:T_1} and \eqref{eqn:dichotomy:T_3} are respectively weaker than their counterparts \cite[(2.2) and (2.5)]{bellazziniStableStandingWaves2011}. It is worth considering weaker conditions because the corresponding limits are not satisfied in the case of the Kirchhoff-type equation in Section \ref{intro:Kirchhoff} (see Lemma \ref{Kirchhoff:dich:lem:cond:4}).
\item
We only consider the abstract minimization problem stated in Section \ref{intro:AMP} under the hypothesis that $T_\alpha'$ is bounded, which takes away the need for a condition similar to \cite[(2.6)]{bellazziniStableStandingWaves2011}.
\end{enumerate}
\end{rmk}

We need the following preliminary lemma to prove Theorem \ref{dich:thm}.

\begin{lem}
\label{dich:lem:3}
Consider the conditions in Theorem \ref{dich:thm}. Suppose that Condition \ref{dich:cond:1} is satisfied and $\rho \in \ooi{0, \infty}$ is such that Conditions \ref{dich:cond:2} and \ref{dich:cond:3} are satisfied. Let
\begin{equation}
\label{dich:eqn:r}
r
=
\inf
\set*{
	t \in \oci{0, \rho}:
	\frac{\I_\alpha \parens{t^2}}{t^2} = \frac{\I_\alpha \parens{\rho^2}}{\rho^2}
}
\in
\cci{0, \rho}.
\end{equation}
Then $r > 0$ and \eqref{eqn:SSC} holds.
\end{lem}
\begin{proof}
\noindent \emph{1. Proof that $r > 0$.}
By contradiction, suppose that $r = 0$. On one hand, it follows from Condition \ref{dich:cond:2} that $\I_\alpha \parens{\rho^2} < 0$. On the other hand, it follows from Condition \ref{dich:cond:1} that
$\frac{\I_\alpha \parens{t^2}}{t^2} \to 0$ as $t \to 0^+$. We just obtained a contradiction.
\\ \\
\noindent \emph{2. A preliminary fact.}
Consider the function
\[
\oci{0, r} \ni t \mapsto f \parens{t}
:=
\frac{\I_\alpha \parens{t^2}}{t^2} \in \ooi{- \infty, 0}.
\]
We want to prove that
\begin{equation}
\label{dich:aux:11}
\set*{t \in \oci{0, r} : f \parens{t} = \min f}
=
\set{r}.
\end{equation}
\noindent \emph{2.1. Proof that $f \parens{r} = \min f$.}
By contradiction, suppose that there exists $r_* \in \ooi{0, r}$ such that
$f \parens{r_*} < f \parens{r} < 0$.
Due to Conditions \ref{dich:cond:1} and \ref{dich:cond:3}, there exists
$r_{**} \in \ooi{0, r_*}$ such that $f \parens{r_{**}} = f \parens{r}$, which contradicts the definition of $r$.
\\ \\
\noindent \emph{2.2. Proof that if $f \parens{t} = \min f$, then $t = r$.}
If we suppose that $0 < t <r$ and $f \parens{t} = \min f$, then we obtain a contradiction with the definition of $r$.
\\ \\
\noindent \emph{3. Proof that \eqref{eqn:SSC} is satisfied.}
In view of \eqref{dich:aux:11},
\[
\frac{\mu^2}{r^2} \I_\alpha \parens{r^2}
<
\I_\alpha \parens{\mu^2}
\quad \text{and} \quad
\frac{r^2 - \mu^2}{r^2} \I_\alpha \parens{r^2}
<
\I_\alpha \parens{r^2 - \mu^2}.
\]
It suffices to sum these inequalities to obtain the strict sub-additivity condition.
\end{proof}

We proceed to the proof of the theorem.

\begin{proof}[Proof of Theorem \ref{dich:thm}]
\emph{1. Definition of $r$.}
Due to Lemma \ref{dich:lem:3},
\[
0
<
r
:=
\inf
\set*{
	t \in \oci{0, \rho}:
	\frac{\I_\alpha \parens{t^2}}{t^2} = \frac{\I_\alpha \parens{\rho^2}}{\rho^2}
}
\leq
\rho.
\]

\noindent\emph{2. Proof that $r = \rho$.}
By contradiction, suppose that $0 < r < \rho$.
In view of Condition \ref{dich:cond:4}, we can use Lemma \ref{dich:lem} to deduce that there exists $u_\infty \in W^{1, 2}_\alpha$ which solves \eqref{eqn:AMP:3}. In particular, $u_\infty \in M \parens{\rho^2} \neq \emptyset$.

Let us show that if $g_{u_\infty}$ is an ASP of $u_\infty$, then $h'_{\alpha, g_{u_\infty}} \parens{1^-} = 0$. Suppose without loss of generality that
$\Theta_{g_{u_\infty}}' \parens{1} > 0$ (the argument is analogous if
$\Theta_{g_{u_\infty}}' \parens{1} < 0$). It follows from the definition of $r$ that there exists $\eps > 0$ such that
\[
\frac{\I_\alpha \parens{r^2}}{r^2}
\leq
\frac{
	\I_\alpha \parens*{\Theta_{g_{u_\infty}} \parens{\theta} r^2}
}{
	\Theta_{g_{u_\infty}} \parens{\theta} r^2
},
\quad \text{and so} \quad
\frac{I_\alpha \parens{u_\infty}}{r^2}
\leq
\frac{
	I_\alpha \parens*{g_{u_\infty} \parens{\theta}}
}{
	\Theta_{g_{u_\infty}} \parens{\theta} r^2
}
\]
for every $\theta \in \oci{1 - \eps, 1}$, so
$h'_{\alpha, g_{u_\infty}} \parens{1^-} = 0$.

The result in the previous paragraph contradicts Condition \ref{dich:cond:5}, so
$r = \rho$.
\\ \\
\noindent\emph{3. Proof that \eqref{eqn:AMP} has a solution.}
We proved that $r = \rho$, so Condition \ref{dich:cond:4} ensures that we can use Lemma \ref{dich:lem} to conclude.
\end{proof}

\section{Proofs of results in Section \ref{intro:applications}}
\label{applications}
\subsection{Ground states of \eqref{eqn:delta-Kirchhoff}}
\label{Kirchhoff}

\subsubsection{Avoiding vanishing}
\label{Kirchhoff:vanishing}

\begin{lem}
\label{Kirchhoff:lem:vanishing}
Given $\rho \in \ooi{0, \infty}$,
\begin{enumerate}
\item
$\I_\alpha \parens{\rho^2} < \I \parens{\rho^2} \leq 0$ (which implies \eqref{cond:vanishing:1}) and
\item
$I_\alpha|_{W^{1, 2}_\alpha \parens{\rho^2}}$ is coercive (that is, \eqref{cond:vanishing:2} is satisfied).
\end{enumerate}
\end{lem}
\begin{proof}
\emph{1. Proof of the first conclusion.}
It follows from \cite[Theorem 1.1]{yeSharpExistenceConstrained2015} that
$\I \parens{\rho^2} \leq 0$ and there exists
$v \in W^{1, 2} \parens{\rho^2}$ that solves \eqref{eqn:AMP:2}. That is, there exists
$\omega \in \real$ such that
\[
-\parens*{1 + \int \abs*{\nabla v \parens{x}}^2 \dif x} \Delta v
+
\omega v
=
v \abs{v}^{p - 2}
\]
in the weak sense. As $p \in \ooi{2, \frac{5}{2}}$, we deduce that $- \Delta v \in L^2$. In particular, $v \in W^{2, 2}$. As $I \parens{v} = I \parens{v \parens{\cdot + x}}$ for every
$x \in \real^3$ and $v \not \equiv 0$, we can suppose that $v \parens{0} \neq 0$. Finally, the result follows from an application of Lemma \ref{lem:calI_alpha<calI}.
\\ \\
\noindent \emph{2. Proof that \eqref{cond:vanishing:2} is satisfied.}
Corollary of Lemma \ref{delta-NLSE:van:lem:cond:2} because
$I_\alpha \parens{u} \geq J_\alpha \parens{u}$ for every $u \in W^{1, 2}_\alpha$.
\end{proof}

\subsubsection{Avoiding dichotomy}
\label{Kirchhoff:dichotomy}

\begin{lem}
\label{Kirchhoff:dich:lem:cond:1}
Condition \ref{dich:cond:1} is satisfied.
\end{lem}
\begin{proof}
On one hand, it is clear that $J_\alpha \parens{u} \leq I_\alpha \parens{u}$ for every
$u \in W^{1, 2}_\alpha$, so
$
\frac{\mathcal{J}_\alpha \parens{\rho^2}}{\rho^2}
\leq
\frac{\mathcal{I}_\alpha \parens{\rho^2}}{\rho^2}
$
for every $\rho \in \ooi{0, \infty}$. On the other hand, it follows from Lemma \ref{Kirchhoff:lem:vanishing} that $\mathcal{I}_\alpha \parens{\rho^2} < 0$ for every
$\rho \in \ooi{0, \infty}$. In this context, the result follows from the Squeeze Theorem and Lemma \ref{delta-NLSE:dich:lem:cond:1}.
\end{proof}

\begin{lem}
\label{Kirchhoff:dich:lem:cond:2}
Condition \ref{dich:cond:2} is satisfied for any
$\rho \in \ooi{0, \infty}$.
\end{lem}
\begin{proof}
It is clear that $J_\alpha \parens{u} \leq I_\alpha \parens{u}$ for every
$u \in W^{1, 2}_\alpha$, so it follows from Lemma \ref{delta-NLSE:van:lem:cond:2} that
$\mathcal{I}_\alpha \parens{r^2} \geq \mathcal{J}_\alpha \parens{r^2} > - \infty$. The other inequality is a corollary of Lemma \ref{Kirchhoff:lem:vanishing}.
\end{proof}

\begin{lem}
\label{Kirchhoff:dich:lem:cond:3}
Condition \ref{dich:cond:3} is satisfied for any
$\rho \in \ooi{0, \infty}$.
\end{lem}
\begin{proof}
We argue similarly as in \cite[Proof of Theorem 4.1, Step 4]{bellazziniScalingPropertiesFunctionals2011}. It suffices to prove that if
$r \in \oci{0, \rho}$ and $\set{r_n}_{n \in \nat} \subset \oci{0, \rho}$ is such that
$r_n \to r$ as $n \to \infty$, then
\[
\I_\alpha \parens{r^2}
\leq
\liminf_{n \to \infty} \I_\alpha \parens{r_n^2}
\quad \text{and} \quad
\limsup_{n \to \infty} \I_\alpha \parens{r_n^2}
\leq
\I_\alpha \parens{r^2}.
\]
\\ \\
\noindent \emph{1. Proof that
$
\I_\alpha \parens{r^2}
\leq
\liminf_{n \to \infty} \I_\alpha \parens{r_n^2}
$.
}
Due to Condition \ref{dich:cond:2}, we can associate each
$n \in \nat$ with a
$u_n = \phi_{\lambda, n} + q_n G_\lambda \in W^{1, 2}_\alpha \parens{r_n^2}$
such that
\[
\I_\alpha \parens{r_n^2}
<
I_\alpha \parens{u_n}
<
\I_\alpha \parens{r_n^2} + \frac{1}{n}
<
\frac{1}{n}.
\]
Consider an $\eps \in \ooi{0, \infty}$ and let $K \in \ooi{0, \infty}$ be furnished by Proposition \ref{prelim:prop:GN} so that
\begin{multline}
\label{SP:eqn:aux:3}
\frac{1}{2} \norm{\phi_{\lambda, n}}_{\dot{W}^{1, 2}}^2
+
\frac{\lambda}{2}
\parens*{
	\norm{\phi_{\lambda, n}}_{L^2}^2 - r_n^2
}
+
\frac{1}{2}
\parens*{\alpha + \frac{\sqrt{\lambda}}{4 \pi}} \abs{q_n}^2
-
\\
-
\frac{K}{p}
\parens*{
	\norm{\phi_{\lambda, n}}_{\dot{W}^{1, 2}}^{\frac{3 \parens{p - 2}}{2}}
	r_n^{\frac{6 - p}{2}}
	+
	\abs{q}^{3 \parens{p - 2}} r_n^{2 \parens{3 - p}}
}
\leq
\\
\leq
\frac{1}{2} H_\alpha \parens{u_n}
-
\frac{1}{p} C \parens{u_n}
\leq
I_\alpha \parens{u_n}
<
\frac{1}{n}
\quad \text{for every} \quad
n \in \nat.
\end{multline}
It follows from \eqref{SP:eqn:aux:3} that $\set{\norm{u_n}_{W^{1, 2}_\alpha}}_{n \in \nat}$ is bounded. We obtain
\begin{multline*}
\I_\alpha \parens{r^2}
\leq
I_\alpha \parens*{\frac{r}{r_n} u_n}
=
\\
=
\frac{1}{2}
\parens*{\frac{r}{r_n}}^2
H_\alpha \parens{u_n}
+
\frac{1}{4}
\parens*{\frac{r}{r_n}}^4
H_\alpha \parens{u_n}^2
-
\frac{1}{p}
\parens*{\frac{r}{r_n}}^p
C \parens{u_n}
\end{multline*}
for every $n \in \nat$, so
$
\I_\alpha \parens{r^2}
\leq
\liminf_{n \to \infty} \I_\alpha \parens{r_n^2}
$.
\\ \\
\noindent \emph{2. Proof that
$
\limsup_{n \to \infty} \I_\alpha \parens{r_n^2}
\leq
\I_\alpha \parens{r^2}
$.
}
Suppose that $\set{u_n}_{n \in \nat}$ is a minimizing sequence o
 $I_\alpha|_{W^{1, 2}_\alpha \parens{r^2}}$. Then
$
\I_\alpha \parens{r_n^2}
\leq
I_\alpha \parens{\parens{r_n / r} u_n}
$
for every $n \in \nat$, which implies
$
\limsup_{n \to \infty} \I_\alpha \parens{r_n^2}
\leq
\I_\alpha \parens{r^2}
$.
\end{proof}

\begin{lem}
\label{Kirchhoff:dich:lem:cond:4}
Condition \ref{dich:cond:4} is satisfied for any $\rho \in \ooi{0, \infty}$.
\end{lem}
\begin{proof}
\emph{1. Construction of the sequence $\parens{u_n}_{n \in \nat}$ and the function $u_\infty \in W^{1, 2}_\alpha \setminus \set{0}$.}
Due to Ekeland's variational principle, we can let $\parens{u_n}_{n \in \nat}$ denote a minimizing sequence of $I_\alpha|_{W^{1, 2}_\alpha \parens{r^2}}$ which is also a Palais--Smale sequence of $I_\alpha|_{W^{1, 2}_\alpha \parens{r^2}}$. Lemma \ref{Kirchhoff:lem:vanishing} shows that $I_\alpha|_{W^{1, 2}_\alpha \parens{r^2}}$ is coercive, so there exists
$u_\infty \in W^{1, 2}_\alpha$ such that, up to subsequence,
$u_n \rightharpoonup u_\infty$ in $W^{1, 2}_\alpha$ as $n \to \infty$. In view of Lemma \ref{Kirchhoff:lem:vanishing}, we can apply Lemma \ref{lem:vanishing} to deduce that $u_\infty \not \equiv 0$. It is clear that if $K$ is a compact subset of $\real^3$, then $W^{1, 2}_\alpha \parens{K}$ is compactly embedded in $L^p \parens{K}$. As such, we can suppose that $\parens{u_n}_{n \in \nat}$ converges a.e. to $u_\infty$ as $n \to \infty$ (which is the case up to subsequence).
\\ \\
\noindent \emph{2. Proof that \eqref{eqn:dichotomy:T_1} holds.}
The Brézis--Lieb Lemma shows that
\begin{equation}
\label{Kirchhoff:eqn:C_1}
C \parens{u_n - u_\infty}
-
C \parens{u_n}
+
C \parens{u_\infty}
\xrightarrow[n \to \infty]{}
0
\end{equation}
(see \cite{brezisRelationPointwiseConvergence1983}). As such, we only have to prove that
\begin{equation}
\label{Kirchhoff:eqn:Ealpha_1}
H_\alpha \parens{u_\infty}^2
\leq
\liminf_{n \to \infty}
\parens*{H_\alpha \parens{u_n}^2 - H_\alpha \parens{u_n - u_\infty}^2}.
\end{equation}
On one hand, Lemma \ref{lem:properties-of-H_alpha} shows that
\[
\parens*{
	H_\alpha \parens{u_n - u_\infty}
	-
	H_\alpha \parens{u_n}
	+
	H_\alpha \parens{u_\infty}
}^2
\xrightarrow[n \to \infty]{}
0.
\]
On the other hand, it is easy to verify that
\begin{multline*}
\parens*{
	H_\alpha \parens{u_n - u_\infty}
	-
	H_\alpha \parens{u_n}
	+
	H_\alpha \parens{u_\infty}
}^2
=
H_\alpha \parens{u_n - u_\infty}^2
+
\\
+
2 H_\alpha \parens{u_n - u_\infty} H_\alpha \parens{u_\infty}
-
2 H_\alpha \parens{u_n}
\parens*{
	H_\alpha \parens{u_n - u_\infty}
	+
	H_\alpha \parens{u_\infty}
}
+
\\
+
H_\alpha \parens{u_n}^2
+
H_\alpha \parens{u_\infty}^2
\end{multline*}
for every $n \in \nat$. It follows that
\[
H_\alpha \parens{u_n - u_\infty}^2
+
2 H_\alpha \parens{u_n - u_\infty} H_\alpha \parens{u_\infty}
-
H_\alpha \parens{u_n}^2
+
H_\alpha \parens{u_\infty}^2
\xrightarrow[n \to \infty]{}
0.
\]
As $H_\alpha$ is nonnegative, this limit implies
\begin{multline*}
H_\alpha \parens{u_\infty}^2
=
\lim_{n \to \infty} \parens*{
	H_\alpha \parens{u_n}^2
	-
	2 H_\alpha \parens{u_n - u_\infty} H_\alpha \parens{u_\infty}
	-
	H_\alpha \parens{u_n - u_\infty}^2
}
\leq
\\
\leq
\liminf_{n \to \infty} \parens*{
	H_\alpha \parens{u_n}^2
	-
	H_\alpha \parens{u_n - u_\infty}^2
},
\end{multline*}
hence the result.
\\ \\
\noindent \emph{3. Proof that \eqref{eqn:dichotomy:T_2} holds.}
Corollary of Lemma \ref{lem:properties-of-H_alpha}.
\\ \\
\noindent \emph{4. Proof that \eqref{eqn:dichotomy:T_3} holds.}
Due to Lemma \ref{delta-NLSE:dich:lem:T_3}, we only have to prove that
\begin{equation}
\label{Kirchhoff:eqn:Ealpha_3}
\liminf_{n, m \to \infty} \parens*{
	B_\alpha ' \parens{u_n} \brackets{u_n - u_m}
	-
	B_\alpha ' \parens{u_m} \brackets{u_n - u_m}
}
\geq
0.
\end{equation}
Due to the Cauchy--Schwarz Inequality,
\begin{multline*}
\parens*{B_\alpha' \parens{u_n} - B_\alpha' \parens{u_m}}
\brackets{u_n - u_m}
=
\\
=
2 \parens*{
	H_\alpha \parens{u_n}^2
	-
	\parens*{
		H_\alpha \parens{u_n} + H_\alpha \parens{u_m}
	}
	\Re \brackets*{
		S_\alpha \parens{u_n, u_m}
	}
	+
	H_\alpha \parens{u_m}^2
}
\geq
\\
\geq
2 \parens*{
	H_\alpha \parens{u_n}^2
	-
	\parens*{
		H_\alpha \parens{u_n} + H_\alpha \parens{u_m}
	}
	\sqrt{H_\alpha \parens{u_n} H_\alpha \parens{u_m}}
	+
	H_\alpha \parens{u_m}^2
},
\end{multline*}
which is nonnegative because
\[
a^4 - \parens{a^2 + b^2} a b + b^4
=
\parens{a - b}^2 \parens{a^2 + ab + b^2}
\geq
0
\]
for every $a, b \geq 0$. It follows that \eqref{Kirchhoff:eqn:Ealpha_3} holds.
\end{proof}

Before proving that Condition \ref{dich:cond:5} holds, let us introduce a notation for a family of ASPs indexed by $\beta \in \real$. Given $u \in W^{1, 2}_\alpha \setminus \set{0}$, we define the ASP
$g_u^\beta \colon \ooi{0, \infty} \to W^{1, 2}_\alpha$ as given by
\begin{equation}
\label{Kirchhoff:eqn:ASP}
g_u^\beta \parens{\theta}
:=
\theta^{\parens*{1 - \frac{3}{2} \beta}}
u \parens{\theta^{- \beta} \cdot},
\end{equation}
so that $\Theta_{g_u^\beta} \parens{\theta} = \theta^2$.

\begin{lem}
\label{Kirchhoff:dich:lem:cond:5}
If $\rho$ is a sufficiently small positive number, then Condition \ref{dich:cond:5} is satisfied.
\end{lem}
\begin{proof}
By contradiction, suppose that Condition \ref{dich:cond:5} is not satisfied for any
$\rho \in \ooi{0, \infty}$. In particular, there exists
$\set{u_n}_{n \in \nat} \subset W^{1, 2}_\alpha$ such that
\begin{itemize}
\item
$r_n := \norm{u_n}_{L^2} \to 0$ as $n \to \infty$;
\item
given $n \in \nat$,
\begin{itemize}
\item
$u_n$ solves \eqref{eqn:AMP:3} in the case $r = r_n$;
\item
if $g_{u_n} \colon \ooi{0, \infty} \to W^{1, 2}_\alpha$ is an ASP of $u_n$, then $h'_{\alpha, g_{u_n}} \parens{1} = 0$.
\end{itemize}
\end{itemize}
In view of Lemma \ref{lem:vanishing}, we can associate each
$n \in \nat$ with a pair
$\parens{\phi_n, q_n} \in W^{1, 2} \times \parens{\complex \setminus \set{0}}$
such that $u_n = \phi_n + q_n G_{\nu_n}$, where
\[
\nu_n
:=
\frac{\abs{q_n}^4}{2^4 \parens{8 \pi}^2 r_n^4}.
\]
An elementary computation shows that
\begin{multline*}
h_{\alpha, g_{u_n}^\beta} \parens{\theta}
=
\frac{1}{2}
\parens*{
	H_\alpha \parens*{g_{u_n}^\beta \parens{\theta}}
	-
	\theta^2 H_\alpha \parens{u_n}
}
+
\\
+
\frac{1}{4}
\parens*{
	H_\alpha \parens*{g_{u_n}^\beta \parens{\theta}}^2
	-
	\theta^2 H_\alpha \parens{u_n}^2
}
-
\\
-
\frac{1}{p}
\parens*{
	\theta^{\parens*{
			\parens*{1 - \frac{3}{2} \beta} p + 3 \beta
	}}
	-
	\theta^2
}
C \parens{u_n},
\end{multline*}
where
\begin{multline*}
H_\alpha \parens*{g_{u_n}^\beta \parens{\theta}}
-
\theta^2 H_\alpha \parens{u_n}
=
\\
=
\begin{bmatrix}
\theta^{\parens{2 - 2 \beta}} - \theta^2
&
0
&
\theta^{\parens{2 - \beta}} - \theta^2
&
\theta^{\parens{2 - 2 \beta}} - \theta^2
\end{bmatrix}
\underbrace{
	\begin{bmatrix}
	\norm{\phi_n}_{\dot{W}^{1, 2}}^2
	\\
	\nu_n \parens*{\norm{\phi_n}_{L^2}^2 - r_n^2}
	\\
	\alpha \abs{q_n}^2
	\\
	\frac{\sqrt{\nu_n}}{4 \pi} \abs{q_n}^2
	\end{bmatrix}
}_{=: x_n}
\end{multline*}
and
\begin{multline*}
H_\alpha \parens*{g_u^\beta \parens{\theta}}^2
-
\theta^2 H_\alpha \parens{u}^2
=
\\
=
x_n^t
\begin{bmatrix}
\theta^{\parens{4 - 4 \beta}} - \theta^{2}
&
\theta^{\parens{4 - 2 \beta}} - \theta^{2}
&
\theta^{\parens{4 - 3 \beta}} - \theta^{2}
&
\theta^{\parens{4 - 4 \beta}} - \theta^{2}
\\
\ast
&
\theta^{4} - \theta^{2}
&
\theta^{\parens{4 - \beta}} - \theta^{2}
&
\theta^{\parens{4 - 2 \beta}} - \theta^{2}
\\
\ast
&
\ast
&
\theta^{\parens{4 - 2 \beta}} - \theta^{2}
&
\theta^{\parens{4 - 3 \beta}} - \theta^{2}
\\
\ast
&
\ast
&
\ast
&
\theta^{\parens{4 - 4 \beta}} - \theta^{2}
\end{bmatrix}
x_n.
\end{multline*}
It follows from the hypothesis that
\begin{multline}
\label{Kirchhoff:eqn:aux:3}
h_{\alpha, g_{u_n}^\beta} ' \parens{1}
=
\frac{1}{2}
\times
\left. \frac{\dif}{\dif \theta} \parens*{
	H_\alpha \parens*{g_{u_n}^\beta \parens{\theta}}
	-
	\theta^2 H_\alpha \parens{u_n}
} \right|_{\theta = 1}
+
\\
+
\frac{1}{4}
\times
\left. \frac{\dif}{\dif \theta} \parens*{
	H_\alpha \parens*{g_{u_n}^\beta \parens{\theta}}^2
	-
	\theta^2 H_\alpha \parens{u_n}^2
} \right|_{\theta = 1}
-
\\
-
\frac{1}{p}
\parens*{\parens*{1 - \frac{3}{2} \beta} p + 3 \beta - 2}
C \parens{u_n}
=
0,
\end{multline}
where
\[
\frac{1}{2}
\times
\left. \frac{\dif}{\dif \theta} \parens*{
	H_\alpha \parens*{g_{u_n}^\beta \parens{\theta}}
	-
	\theta^2 H_\alpha \parens{u_n}
} \right|_{\theta = 1}
=
-
\beta \norm{\phi_n}_{\dot{W}^{1, 2}}^2
-
\beta \frac{\alpha}{2} \abs{q_n}^2
-
\beta \frac{\sqrt{\nu_n}}{4 \pi} \abs{q_n}^2
\]
and
\begin{multline*}
\left. \frac{\dif}{\dif \theta} \parens*{
	H_\alpha \parens*{g_u^\beta \parens{\theta}}^2
	-
	\theta^2 H_\alpha \parens{u}^2
} \right|_{\theta = 1}
=
\\
=
x_n^t
\begin{bmatrix}
2 - 4 \beta
&
2 - 2 \beta
&
2 - 3 \beta
&
2 - 4 \beta
\\
\ast
&
2
&
2 - \beta
&
2 - 2 \beta
\\
\ast
&
\ast
&
2 - 2 \beta
&
2 - 3 \beta
\\
\ast
&
\ast
&
\ast
&
2 - 4 \beta
\end{bmatrix}
x_n.
\end{multline*}
The fact that \eqref{Kirchhoff:eqn:aux:3} holds for every
$\beta \in \real$ is equivalent to \eqref{Kirchhoff:eqn:aux:4} and \eqref{Kirchhoff:eqn:aux:5}, where
\begin{equation}
\label{Kirchhoff:eqn:aux:4}
\frac{1}{2} H_\alpha \parens{u_n}^2 = \frac{p - 2}{p} C \parens{u_n};
\end{equation}
\begin{equation}
\label{Kirchhoff:eqn:aux:5}
\frac{3}{2} \times \frac{p - 2}{p} C \parens{u_n}
=
\norm{\phi_n}_{\dot{W}^{1, 2}}^2
+
\frac{\alpha}{2} \abs{q_n}^2
+
\frac{\sqrt{\nu_n}}{4 \pi} \abs{q_n}^2
+
\frac{1}{4} A_n
\end{equation}
and
\[
A_n
:=
x_n^t
\begin{bmatrix}
4	 & 2 & 3 & 4
\\
\ast	 & 0 & 1 & 2
\\
\ast & \ast & 2 & 3
\\
\ast & \ast & \ast & 4
\end{bmatrix}
x_n.
\]
As $u_n = \phi_n + q_n G_{\nu_n}$, the Triangle Inequality implies
\begin{multline}
\label{Kirchhoff:eqn:triangular}
\nu_n \abs*{\norm{\phi_n}_{L^2}^2 - r_n^2}
=
\nu_n
\abs*{
	\parens*{\norm{\phi_n}_{L^2} - r_n}
	\parens*{\norm{\phi_n}_{L^2} + r_n}
}
\leq
\\
\leq
\nu_n
\abs{q_n} \norm{G_{\nu_n}}_{L^2}
\parens*{\abs{q_n} \norm{G_{\nu_n}}_{L^2} + 2 r_n}
=
8 \nu_n r_n^2
=
\frac{\sqrt{\nu_n}}{4 \pi} \abs{q_n}^2.
\end{multline}

We claim that
\begin{equation}
\label{Kirchhoff:eqn:An_geq_0}
A_n \geq 0.
\end{equation}
Indeed, it follows from \eqref{Kirchhoff:eqn:triangular} that
\begin{multline*}
A_n
=
x_n^t
\begin{bmatrix}
4 & 0 & 3 & 4
\\
\ast & 0 & 0 & 0
\\
\ast & \ast & 2 & 3
\\
\ast & \ast & \ast & 4
\end{bmatrix}
x_n
+
4
\norm{\phi_n}_{\dot{W}^{1, 2}}^2
\nu_n \parens*{\norm{\phi_n}_{L^2}^2 - r_n^2}
+
\\
+
2
\alpha \abs{q_n}^2
\nu_n \parens*{\norm{\phi_n}_{L^2}^2 - r_n^2}
+
4
\nu_n \parens*{\norm{\phi_n}_{L^2}^2 - r_n^2}
\frac{\sqrt{\nu_n}}{4 \pi} \abs{q_n}^2
\geq
\\
\geq
x_n^t
\begin{bmatrix}
4 & 0 & 3 & 4
\\
\ast & 0 & 0 & 0
\\
\ast & \ast & 2 & 3
\\
\ast & \ast & \ast & 4
\end{bmatrix}
x_n
-
4
\norm{\phi_n}_{\dot{W}^{1, 2}}^2
\frac{\sqrt{\nu_n}}{4 \pi} \abs{q_n}^2
-
2
\alpha \abs{q_n}^2
\frac{\sqrt{\nu_n}}{4 \pi} \abs{q_n}^2
-
\\
-
4
\parens*{\frac{\sqrt{\nu_n}}{4 \pi} \abs{q_n}^2}^2
=
x_n^t
\begin{bmatrix}
4 & 0 & 3 & 2
\\
\ast & 0 & 0 & 0
\\
\ast & \ast & 2 & 2
\\
\ast & \ast & \ast & 0
\end{bmatrix}
x_n
\geq
0.
\end{multline*}

It follows from \eqref{Kirchhoff:eqn:aux:4} and \eqref{Kirchhoff:eqn:aux:5} that
\begin{multline}
\label{Kirchhoff:eqn:aux:7}
I_\alpha \parens{u_n}
=
-
\frac{14 - 5 p}{6 \parens{p - 2}}
\norm{\phi_n}_{\dot{W}^{1, 2}}^{2}
+
\frac{\nu_n}{2}
\parens*{
	\norm{\phi_n}_{L^{2}}^{2}
	-
	r_n^2
}
-
\frac{5 - 2 p}{3 \parens{p - 2}} \alpha \abs{q_n}^2
-
\\
-
\frac{14 - 5 p}{6 \parens{p - 2}}
\cdot
\frac{\sqrt{\nu_n}}{4 \pi} \abs{q_n}^2
-
\frac{4 - p}{3 \parens{p - 2}} A_n.
\end{multline}

Now, we claim that
\begin{equation}
\label{Kirchhoff:eqn:limit}
\norm{\phi_n}_{\dot{W}^{1, 2}}, ~ \frac{\abs{q_n}^2}{r_n}
\xrightarrow[n \to \infty]{}
0.
\end{equation}
Indeed, it follows from \eqref{Kirchhoff:eqn:triangular} and \eqref{Kirchhoff:eqn:aux:7} that
\begin{multline*}
\I_\alpha \parens{r_n^2}
\leq
-
\frac{14 - 5 p}{6 \parens{p - 2}}
\norm{\phi_n}_{\dot{W}^{1, 2}}^{2}
-
\frac{5 - 2 p}{3 \parens{p - 2}} \alpha \abs{q_n}^2
-
\\
-
\frac{2}{3} \times \frac{5 - 2 p}{p - 2}
\frac{\sqrt{\nu_n}}{4 \pi} \abs{q_n}^2
-
\frac{4 - p}{3 \parens{p - 1}} A_n
\leq
0.
\end{multline*}
The result follows from Condition \ref{dich:cond:1}.

Inequality \eqref{Kirchhoff:eqn:triangular} implies
\[
H_\alpha \parens{u_n}^2
\leq
\parens*{
	\norm{\phi_n}_{\dot{W}^{1, 2}}^2
	+
	\alpha \abs{q_n}^2
	+
	2 \frac{\sqrt{\nu_n}}{4 \pi} \abs{q_n}^2
}^2.
\]
As such, it follows from \eqref{Kirchhoff:eqn:aux:4}, \eqref{Kirchhoff:eqn:aux:5} and \eqref{Kirchhoff:eqn:An_geq_0} that
\begin{multline*}
\norm{\phi_n}_{\dot{W}^{1, 2}}^2
+
\alpha \abs{q_n}^2
+
\frac{\sqrt{\nu_n}}{4 \pi} \abs{q_n}^2
\leq
\frac{3}{4} H_\alpha \parens{u_n}^2
\leq
\\
\leq
\frac{3}{4}
\parens*{
	\norm{\phi_n}_{\dot{W}^{1, 2}}^2
	+
	\alpha \abs{q_n}^2
	+
	2 \frac{\sqrt{\nu_n}}{4 \pi} \abs{q_n}^2
}^2
\end{multline*}
for every $n \in \nat$, which contradicts \eqref{Kirchhoff:eqn:limit}.
\end{proof}

\subsection{Ground states of \eqref{eqn:reduced-delta-SP}}
\label{SP}

\subsubsection{Avoiding vanishing}

\begin{lem}
\label{SP:lem:vanishing}
If $\rho$ is a sufficiently small positive number, then
\begin{enumerate}
\item
$\I_\alpha \parens{\rho^2} < \I \parens{\rho^2} < 0$ (which implies \eqref{cond:vanishing:1});
\item
\eqref{cond:vanishing:2} is satisfied.
\end{enumerate}
\end{lem}
\begin{proof}
\emph{1. Proof that $\I_\alpha \parens{\rho^2} < \I \parens{\rho^2} < 0$.}
Due to \cite[Step 2 in Proof of Theorem 4.1]{bellazziniScalingPropertiesFunctionals2011},
$\I \parens{\rho^2} \in \ooi{- \infty, 0}$ for every $\rho \in \ooi{0, \infty}$.
It follows from \cite[Theorem 4.1]{bellazziniScalingPropertiesFunctionals2011} that if
$\rho$ is a sufficiently small positive number, then there exists
$v \in W^{1, 2} \parens{\rho^2}$ that solves \eqref{eqn:AMP:2}. We know that $\Phi_{\abs{v}^2} v \in L^2$, so we can argue as in
the proof of Lemma \ref{Kirchhoff:lem:vanishing} to conclude.
\\ \\
\noindent \emph{2. Proof that \eqref{cond:vanishing:2} is satisfied.}
It suffices to argue as in the proof of Lemma \ref{Kirchhoff:lem:vanishing}.
\end{proof}

\subsubsection{Avoiding dichotomy}

The following lemma is proved with the arguments in the proofs of Lemmas \ref{Kirchhoff:dich:lem:cond:1}--\ref{Kirchhoff:dich:lem:cond:3}.

\begin{lem}
\label{SP:dich:lem:cond:1-3}
\begin{enumerate}
\item
Condition \ref{dich:cond:1} is satisfied.
\item
If $\rho$ is sufficiently small positive number, then Conditions \ref{dich:cond:2} and \ref{dich:cond:3} are satisfied.
\end{enumerate}
\end{lem}

We need a preliminary lemma before proceeding to Condition \ref{dich:cond:4}. The next result is similar to \cite[Lemma 2.2 (i)]{zhaoExistenceSolutionsSchrodinger2008} and may be proved analogously.

\begin{lem}
\label{SP:dich:lem:ZhaoZhao}
If $u_n \rightharpoonup u_\infty$ in $W^{1, 2}_\alpha$ and $u_n \to u_\infty$ a.e. as
$n \to \infty$, then
\[
B \parens{u_n - u_\infty} - B \parens{u_n} + B \parens{u_\infty}
\xrightarrow[n \to \infty]{}
0.
\]
\end{lem}

\begin{lem}
\label{SP:dich:lem:cond:4}
Condition \ref{dich:cond:4} is satisfied for every $\rho \in \ooi{0, \infty}$.
\end{lem}
\begin{proof}
\emph{1. Construction of the sequence $\parens{u_n}_{n \in \nat}$ and the function $u_\infty \in W^{1, 2}_\alpha \setminus \set{0}$.}
It suffices to proceed as in the proof of Lemma \ref{Kirchhoff:dich:lem:cond:3}.
\\ \\
\noindent \emph{2. Proof that \eqref{eqn:dichotomy:T_1} holds.}
Corollary of the Brézis--Lieb Lemma and Lemma \ref{SP:dich:lem:ZhaoZhao}.
\\ \\
\noindent \emph{3. Proof that \eqref{eqn:dichotomy:T_2} holds.}
Corollary of the fact that both nonlinear functionals $B$ and $C$ are positively homogeneous.
\\ \\
\noindent \emph{4. Proof that \eqref{eqn:dichotomy:T_3} holds.}
Due to Lemma \ref{delta-NLSE:dich:lem:T_3}, we only have to prove that
\[
\parens*{B ' \parens{u_n} - B ' \parens{u_m}} \brackets{u_n - u_m}
\to
0
\quad \text{as} \quad
n, m \to \infty.
\]
It follows from Hölder's Inequality that
\[
B' \parens{u_n} \brackets{u_n - u_m}
\leq
\norm*{\Phi_{\abs{u_n}^2}}_{L^{\frac{2 p}{p - 2}}}
\norm{u_n}_{L^p}
\norm{u_n - u_m}_{L^2}.
\]
In view of the Hardy--Littlewood--Sobolev Inequality, there exists
$c \in \ooi{0, \infty}$ such that
\[
B' \parens{u_n} \brackets{u_n - u_m}
\leq
c
\norm{u_n}_{L^{\frac{12 p}{7 p - 6}}}^2
\norm{u_n}_{L^p}
\norm{u_n - u_m}_{L^2}
\]
for every $n \in \nat$. Let us show that
$B' \parens{u_n} \brackets{u_n - u_m} \to 0$
as $n, m \to \infty$. On one hand, we know that
$I_\alpha|_{W^{1, 2}_\alpha \parens{\rho_0^2}}$ is coercive, so
$\set{u_n}_{n \in \nat}$ is bounded in $W^{1, 2}_\alpha$. As
\[
2 < p + 1 < \frac{5}{2} < \frac{12 p}{7 p - 6} < 3,
\]
we deduce that $\set{u_n}_{n \in \nat}$ is also bounded in
$L^{p + 1} \cap L^{\frac{12 p}{7 p - 6}}$. On the other hand, we have $u_n \rightharpoonup u_\infty$ in $L^2$ as
$n \to \infty$ and $\norm{u_\infty}_{L^2} = \rho_0$, so
$\lim_{n \to \infty} \norm{u_n - u_m}_{L^2} = 0$, hence the result.
\end{proof}

\begin{lem}
\label{SP:dich:lem:cond:5}
If $\rho$ is a sufficiently small positive number, then Condition \ref{dich:cond:5} is satisfied.
\end{lem}
\begin{proof}
By contradiction, suppose that Condition \ref{dich:cond:5} is not satisfied for any
$\rho \in \ooi{0, \infty}$. In particular, there exists
$\set{u_n}_{n \in \nat} \subset W^{1, 2}_\alpha$ such that
\begin{itemize}
\item
$r_n := \norm{u_n}_{L^2} \to 0$ as $n \to \infty$;
\item
given $n \in \nat$,
\begin{itemize}
\item
$u_n$ solves \eqref{eqn:AMP:3} in the case $r = r_n$;
\item
if $g_{u_n} \colon \ooi{0, \infty} \to W^{1, 2}_\alpha$ is an ASP of $u_n$, then $h'_{\alpha, g_{u_n}} \parens{1} = 0$.
\end{itemize}
\end{itemize}
Given $n \in \nat$, let
$\parens{\phi_n, q_n} \in W^{1, 2} \times \parens{\complex \setminus \set{0}}$
be such that $u_n = \phi_n + q_n G_{\nu_n}$, where
\[
\nu_n
:=
\frac{\abs{q_n}^4}{\parens{8 \pi}^2 r_n^4}.
\]
With respect to the ASP \eqref{Kirchhoff:eqn:ASP},
\begin{multline}
\label{SP:eqn:aux:4}
h_{\alpha, g_{u_n}^\beta}' \parens{1}
=
-
\beta
\norm{\phi_n}_{\dot{W}^{1, 2}}^2
-
\frac{\beta}{2} \alpha \abs{q_n}^2
-
\beta \frac{\sqrt{\nu_n}}{4 \pi} \abs{q_n}^2
+
\frac{2 - \beta}{4}
B \parens{u_n}
-
\\
-
\frac{1}{p}
\parens*{\parens*{1 - \frac{3}{2} \beta} p + 3 \beta - 2}
C \parens{u_n}
=
0.
\end{multline}
The fact that \eqref{SP:eqn:aux:4} holds for every
$\beta \in \real$ is equivalent to the two following equalities:
\begin{equation}
\label{SP:eqn:aux:scaling_1}
-
\norm{\phi_n}_{\dot{W}^{1, 2}}^2
-
\frac{\alpha}{2} \abs{q_n}^2
-
\frac{\sqrt{\nu_n}}{4 \pi} \abs{q_n}^2
-
\frac{1}{4} B \parens{u_n}
+
\frac{3}{2} \times \frac{p - 2}{p} C \parens{u_n}
=
0
\end{equation}
and
\begin{equation}
\label{SP:eqn:aux:scaling_2}
\frac{1}{2} B \parens{u_n}
-
\frac{p - 2}{p} C \parens{u_n}
=
0.
\end{equation}
In view of \eqref{SP:eqn:aux:scaling_1} and \eqref{SP:eqn:aux:scaling_2}, we obtain
\begin{multline}
\label{SP:eqn:aux:7}
I_\alpha \parens{u_n}
=
-
\frac{3 - p}{p - 2}
\norm{\phi_n}_{\dot{W}^{1, 2}}^2
+
\frac{\nu_n}{2} \parens*{\norm{\phi_n}_{L^2}^2 - r_n^2}
-
\frac{8 - 3 p}{4 \parens{p - 2}} \alpha \abs{q_n}^2
-
\\
-
\frac{3 - p}{p - 2}
\times
\frac{\sqrt{\nu_n}}{4 \pi} \abs{q_n}^2.
\end{multline}
As $u_n = \phi_n + q_n G_{\nu_n}$, the Triangle Inequality implies
\begin{multline}
\label{SP:eqn:triangular}
\frac{\nu_n}{2} \abs*{\norm{\phi_n}_{L^2}^2 - r_n^2}
=
\frac{\nu_n}{2}
\abs*{
	\parens*{\norm{\phi_n}_{L^2} - r_n}
	\parens*{\norm{\phi_n}_{L^2} + r_n}
}
\leq
\\
\leq
\frac{\nu_n}{2}
\abs{q_n} \norm{G_{\nu_n}}_{L^2}
\parens*{\abs{q_n} \norm{G_{\nu_n}}_{L^2} + 2 r_n}
=
\frac{3}{2} \nu_n r_n^2
=
\frac{3}{4} \times \frac{\sqrt{\nu_n}}{4 \pi} \abs{q_n}^2.
\end{multline}

Let us prove that
\begin{equation}
\label{SP:eqn:aux:tends_to_zero}
\norm{\phi_n}_{\dot{W}^{1, 2}}, ~ \frac{\abs{q_n}^2}{r_n},
~ B \parens{u_n}, ~ C \parens{u_n}
\xrightarrow[n \to \infty]{}
0.
\end{equation}
Indeed, it follows from \eqref{SP:eqn:aux:7} and \eqref{SP:eqn:triangular} that
\[
\I_\alpha \parens{r_n^2}
\leq
-
\frac{3 - p}{p - 2} \norm{\phi_n}_{\dot{W}^{1, 2}}^2
-
\frac{8 - 3 p}{4 \parens{p - 2}} \alpha \abs{q_n}^2
-
\frac{18 - 7 p}{4 \parens{p - 2}}
\times
\frac{\sqrt{\nu_n}}{4 \pi} \abs{q_n}^2
\leq
0.
\]
The limits $\frac{\abs{q_n}^2}{r_n}$, $\norm{\phi_n}_{\dot{W}^{1, 2}} \to 0$ as
$n \to \infty$ follow from Condition \ref{dich:cond:1}. The other limits follow from \eqref{SP:eqn:aux:scaling_1} and \eqref{SP:eqn:aux:scaling_2}.
\\ \\
\noindent \emph{1. Case $2 < p < \frac{12}{5}$.}
Due to Corollary \ref{cor:HLS}, there exists $K \in \ooi{0, \infty}$ such that
$B \parens{u_n} \leq K \norm{u_n}_{L^{\frac{12}{5}}}^4$ for every
$n \in \nat$. By interpolation,
\[
\norm{u_n}_{L^{\frac{12}{5}}}
\leq
\norm{u_n}_{L^p}^{1 - \gamma}
\norm{u_n}_{L^{\frac{13}{5}}}^{\gamma}
\]
for every $n \in \nat$, where
$\gamma := \frac{13}{12} \times \frac{12 - 5 p}{13 - 5 p}$.
In view of \eqref{SP:eqn:aux:scaling_2},
\[
\frac{p - 2}{p} \norm{u_n}_{L^p}^p
=
\frac{1}{2} B \parens{u_n}
\leq
K
\norm{u_n}_{L^p}^{4 \parens{1 - \gamma}}
\norm{u_n}_{L^{\frac{13}{5}}}^{4 \gamma}
\]
for every $n \in \nat$. It is easy to verify that
$4 \parens{1 - \gamma} - p = \frac{20 p - 39}{3 \parens{13 - 5 p}} > 0$ because $p > 2$. As such, we obtain a contradiction with \eqref{SP:eqn:aux:tends_to_zero}.
\\ \\
\noindent \emph{2. Case $p = \frac{12}{5}$.}
In view of \eqref{SP:eqn:aux:scaling_2}, we have
\[
\frac{1}{6} \norm{u_n}_{L^{\frac{12}{5}}}^{\frac{12}{5}}
=
\frac{1}{6} C \parens{u_n}
=
\frac{1}{2} B \parens{u_n}
\leq
c_1 \norm{u_n}_{L^{\frac{12}{5}}}^4
\]
for every $n \in \nat$, which contradicts \eqref{SP:eqn:aux:tends_to_zero}.
\\ \\
\noindent \emph{3. Case $\frac{12}{5} < p < \frac{5}{2}$.}
By interpolation,
\[
\norm{u_n}_{L^{\frac{12}{5}}}^4
\leq
r_n^{4 \parens{1 - \gamma}}
\norm{u_n}_{L^p}^{4 \gamma}
\]
for every $n \in \nat$, where
$\gamma := \frac{p}{6 \parens{p - 2}}$. We can use \eqref{SP:eqn:aux:scaling_2} to obtain a contradiction as in the previous cases.
\end{proof}


\appendix
\section{Properties of the energy functional associated with the $\delta$-NLSE}

The goal of this appendix is to prove that certain conditions of the results in Section \ref{intro:existence-of-soln} are verified in the case $I_\alpha = J_\alpha$, where the energy functional
$J_\alpha \colon W^{1, 2}_\alpha \to \real$ is defined as
\[
J_\alpha \parens{u}
=
\frac{1}{2} H_\alpha \parens{u} - \frac{1}{p} C \parens{u};
\]
$C \colon L^p \to \coi{0, \infty}$ is defined as $C \parens{u} = \norm{u}_{L^p}^p$ and it will be useful to consider the energy level
\[
\J_\alpha \parens{\rho^2}
:=
\inf_{W^{1, 2}_\alpha \parens{\rho^2}} J_\alpha
\in
\coi{- \infty, \infty}.
\]
Notice that $J_\alpha$ is the energy functional whose critical points are associated with standing waves of the $\delta$-NLSE \eqref{intro:eqn:NLSE}. The first result is \cite[Proposition III.1]{adamiExistenceStructureRobustness2022}, but we include its proof for the convenience of the reader.
\begin{lem}
\label{delta-NLSE:van:lem:cond:2}
Given $\rho \in \ooi{0, \infty}$,
$J_\alpha|_{W^{1, 2}_\alpha \parens{\rho^2}}$ is coercive and
$\J_\alpha \parens{\rho^2} > - \infty$.
\end{lem}
\begin{proof}
\emph{1. Proof that $J_\alpha|_{W^{1, 2}_\alpha \parens{\rho^2}}$ is coercive.}
It suffices to prove that both
\[
J|_{W^{1, 2} \parens{\rho^2}}
\quad \text{and} \quad
J_\alpha|_{W^{1, 2}_\alpha \parens{\rho^2} \setminus W^{1, 2} \parens{\rho^2}}
\]
are coercive, where $J := J_\alpha|_{W^{1, 2}}$.
\\ \\
\noindent \emph{1.1. Proof that $J|_{W^{1, 2}} \parens{\rho^2}$ is coercive.}
Due to the Gagliardo--Nirenberg Inequality, there exists $K \in \ooi{0, \infty}$ such that
\begin{multline}
\label{delta-NLSE:aux:1}
J \parens{\phi}
\geq
\frac{1}{2} \norm{\phi}_{\dot{W}^{1, 2}}^2
-
\frac{1}{p} C \parens{\phi}
\geq
\frac{1}{2} \norm{\phi}_{\dot{W}^{1, 2}}^2
-
K
\rho^{\frac{6 - p}{2}}
\norm{\phi}_{\dot{W}^{1, 2}}^{\frac{3 \parens{p - 2}}{2}}
\\
\text{for every} \quad
\phi \in W^{1, 2} \parens{\rho^2}.
\end{multline}
\\ \\
\noindent \emph{1.2. Proof that
$
J_\alpha|_{
	W^{1, 2}_\alpha \parens{\rho^2} \setminus W^{1, 2} \parens{\rho^2}
}
$ is coercive.}
Let $\eps = \frac{1}{\parens{8 \pi}^2}$. It follows from Lemma \ref{lem:H_alpha} and Proposition \ref{prelim:prop:GN} that there exists $K \in \ooi{0, \infty}$ such that
\begin{multline*}
J_\alpha \parens{u}
\geq
\frac{1}{2} \norm{\phi}_{\dot{W}^{1, 2}}^2
+
\frac{\abs{q}^4}{2 \parens{8 \pi \rho}^2}
\parens*{1 + \frac{\norm{\phi}_{L^2}^2}{\rho^2}}
-
\\
-
K
\parens*{
	\rho^{\frac{6 - p}{2}}
	\norm{\phi}_{\dot{W}^{1, 2}}^{\frac{3 \parens{p - 2}}{2}}
	+
	\rho^{2 \parens{3 - p}}
	\abs{q}^{3 \parens{p - 2}}
}
\end{multline*}
for every
$
u
=
\phi + q G_{\eps \abs{q}^4 / \rho^4}
\in
W^{1, 2}_\alpha \parens{\rho^2} \setminus W^{1, 2} \parens{\rho^2}
$.
\\ \\
\emph{2. Proof that $\J_\alpha \parens{\rho^2} > - \infty$.}
It suffices to prove that
\[
\J \parens{\rho^2}
:=
\inf_{W^{1, 2} \parens{\rho^2}} J > - \infty
\quad \text{and} \quad
\inf_{W^{1, 2}_\alpha \parens{\rho^2} \setminus W^{1, 2} \parens{\rho^2}}
	J_\alpha
>
- \infty.
\]
\\ \\
\noindent \emph{2.1. Proof that $\J \parens{\rho^2} > - \infty$.}
In view of \eqref{delta-NLSE:aux:1}, this is a corollary of the fact that
\[
W^{1, 2} \parens{\rho^2} \ni \phi
\mapsto
\frac{1}{2} \norm{\phi}_{\dot{W}^{1, 2}}^2
-
K
\rho^{\frac{6 - p}{2}}
\norm{\phi}_{\dot{W}^{1, 2}}^{\frac{3 \parens{p - 2}}{2}}
\in
\real
\]
is bounded below.
\\ \\
\noindent \emph{2.2. Proof that
$
\inf_{u \in W^{1, 2}_\alpha \parens{\rho^2} \setminus W^{1, 2} \parens{\rho^2}}
	J_\alpha \parens{u}
>
- \infty
$.}
The proof is similar to the previous one.
\end{proof}

The next result shows that the analog to Condition \ref{dich:cond:1} in Theorem \ref{dich:thm} is satisfied.

\begin{lem}
\label{delta-NLSE:dich:lem:cond:1}
The limit $\frac{\mathcal{J}_\alpha \parens{\rho^2}}{\rho^2} \to 0$ as $\rho \to 0^+$ is satisfied.
\end{lem}
\begin{proof}
Let us argue as in the proof of the similar result \cite[Lemma A.1]{bellazziniScalingPropertiesFunctionals2011}.
\\ \\
\noindent \emph{1. Proof of the lemma.}
In view of \cite[Theorem I.2]{adamiExistenceStructureRobustness2022}, we can associate each $\rho \in \ooi{0, \infty}$ with a pair
$
\parens{u_\rho, \omega_\rho}
\in
W^{1, 2}_\alpha \parens{\rho^2} \times \real
$
such that $u_\rho$ is a minimizer of $J_\alpha|_{W^{1, 2}_\alpha \parens{\rho^2}}$ and
$- \omega_\rho$ is its associated Lagrange multiplier, that is,
$J_\alpha \parens{u_\rho} = \mathcal{J}_\alpha \parens{\rho^2}$
and
\[
J_\alpha' \parens{u_\rho} \brackets{v}
=
-
\omega_\rho
\Re \brackets*{
	\int \overline{u_\rho \parens{x}} v \parens{x} \dif x
}
\]
for every $v \in W^{1, 2}_\alpha$. At this point, we only have to prove that the following limit is satisfied:
\begin{equation}
\label{eqn:omega}
\omega_\rho \xrightarrow[\rho \to 0^+]{} 0.
\end{equation}
Indeed, it follows from \cite[Proposition III.2]{adamiExistenceStructureRobustness2022} that
\[
-\frac{\omega_\rho}{2}
=
\frac{H_\alpha \parens{u_\rho} - C \parens{u_\rho}}{2 \rho^2}
\leq
\frac{1}{\rho^2}
\parens*{
	\frac{1}{2} H_\alpha \parens{u_\rho}
	-
	\frac{1}{p} C \parens{u_\rho}
}
=
\frac{J_\alpha \parens{u_\rho}}{\rho^2}
<
0
\]
for every $\rho \in \ooi{0, \infty}$, so the lemma follows from \eqref{eqn:omega}.
\\ \\
\noindent \emph{2. Proof of \eqref{eqn:omega}.}
By contradiction, suppose that there exists $c_1 \in \ooi{0, 1}$ and
$\set{\rho_n}_{n \in \nat} \subset \ooi{0, \infty}$ such that
$\rho_n \to 0$ as $n \to \infty$ and $\omega_n := \omega_{\rho_n} \geq c_1$ for every
$n \in \nat$. 

Given $n \in \nat$, let $u_n = u_{\rho_n} \in W^{1, 2}_\alpha \parens{\rho_n^2}$.
In view of Lemma \ref{lem:vanishing},
\[u_n \in W^{1, 2}_\alpha \parens{\rho_n^2} \setminus W^{1, 2} \parens{\rho_n^2}\]
for every $n \in \nat$. As such, we can associate each $n \in \nat$ with a pair
\[\parens{\phi_n, q_n} \in W^{1, 2} \times \parens{\complex \setminus \set{0}}\]
such that $u_n = \phi_n + q_n G_{\nu_n}$, where
$\nu_n := \eps \frac{\abs{q_n}^4}{\rho_n^4} > 0$ and
$\eps := \frac{1}{\parens{8 \pi}^2}$.
In particular, there exists $c_3 \in \ooi{0, \infty}$ such that
\begin{equation}
\label{Kirchhoff:aux:1}
\norm{u_n}_{W^{1, 2}_\alpha}^{p - 2} \geq c_3
\quad \text{for every} \quad
n \in \nat.
\end{equation}

Let $K_\eps \in \ooi{0, \infty}$ be furnished by Proposition \ref{prelim:prop:GN}. In view of Lemma \ref{lem:H_alpha},
\begin{multline}
\label{Kirchhoff:eqn:coercive}
0
>
J_\alpha \parens{u_n}
\geq
\frac{1}{2} \norm{\phi_n}_{\dot{W}^{1, 2}}^2
+
\frac{\abs{q_n}^4}{2 \parens{8 \pi \rho_n}^2}
\parens*{1 + \frac{\norm{\phi_n}_{L^2}^2}{\rho_n^2}}
-
\frac{1}{p} \norm{u_n}_{L^p}^p
\geq
\\
\geq
\parens*{
	\frac{1}{2} \norm{\phi_n}_{\dot{W}^{1, 2}}^2
	-
	\frac{K_\eps}{p}
	\norm{\phi_n}_{\dot{W}^{1, 2}}^{\frac{3 \parens{p - 2}}{2}}
	\rho_n^{\frac{6 - p}{2}}
}
+
\\
+
\parens*{
	\frac{\abs{q_n}^4}{2 \parens{8 \pi \rho_n}^2}
	\parens*{1 + \frac{\norm{\phi_n}_{L^2}^2}{\rho_n^2}}
	-
	\frac{K_\eps}{p}
	\abs{q_n}^{3 \parens{p - 2}}
	\rho_n^{2 \parens{3 - p}}
}
\quad \text{for every} \quad n \in \nat.
\end{multline}

We claim that $\liminf_{n \to \infty} \norm{\phi_n}_{\dot{W}^{1, 2}} = 0$. By contradiction, suppose that $\liminf_{n \to \infty} \norm{\phi_n}_{\dot{W}^{1, 2}} > 0$. Due to the limit
$\rho_n \to 0$ as $n \to \infty$,
\[
\liminf_{n \to \infty}
\parens*{
	\frac{\abs{q_n}^4}{2 \parens{8 \pi \rho_n}^2}
	\parens*{1 + \frac{\norm{\phi_n}_{L^2}^2}{\rho_n^2}}
	-
	\frac{K_\eps}{p}
	\abs{q_n}^{3 \parens{p - 2}}
	\rho_n^{2 \parens{3 - p}}
}
\geq
0
\]
and
\[
\liminf_{n \to \infty}
\parens*{
	\frac{1}{2} \norm{\phi_n}_{\dot{W}^{1, 2}}^2
	-
	\frac{K_\eps}{p}
	\norm{\phi_n}_{\dot{W}^{1, 2}}^{\frac{3 \parens{p - 2}}{2}}
	\rho_n^{\frac{6 - p}{2}}
}
>
0.
\]
We obtained a contradiction with \eqref{Kirchhoff:eqn:coercive}, so
$\liminf_{n \to \infty} \norm{\phi_n}_{\dot{W}^{1, 2}} = 0$. A similar argument shows that
$\liminf_{n \to \infty} \frac{\abs{q_n}^2}{\rho_n} = 0$. In particular, it follows from Lemma \ref{lem:H_alpha} that $\liminf_{n \to \infty} \norm{u_n}_{W^{1, 2}_\alpha} = 0$. This contradicts \eqref{Kirchhoff:aux:1} and the proof is finished.
\end{proof}

Our last result is that the analog to \eqref{eqn:dichotomy:T_3} will also be satisfied.

\begin{lem}
\label{delta-NLSE:dich:lem:T_3}
\sloppy
If $r \in \ooi{0, \infty}$ and $\parens{u_n}_{n \in \nat}$ is a Palais--Smale sequence of $I_\alpha|_{W^{1, 2}_\alpha \parens{r^2}}$ such that
\[
I_\alpha \parens{u_n} \xrightarrow[n \to \infty]{} \I_\alpha \parens{r^2}
\quad \text{and} \quad
u_n \xrightharpoonup[n \to \infty]{W^{1, 2}_\alpha} u_\infty \not \equiv 0,
\]
then
\[
\parens*{C ' \parens{u_n} - C ' \parens{u_m}} \brackets{u_n - u_m}
\xrightarrow[n, m \to \infty]{}
0.
\]
\end{lem}
\begin{proof}
In view of Hölder's Inequality,
\begin{multline*}
\abs*{
	C' \parens{u_n} \brackets{u_n - u_m}
	-
	C' \parens{u_m} \brackets{u_n - u_m}
}
\leq
\\
\leq
\norm*{u_n \abs{u_n}^{p - 2} - u_m \abs{u_m}^{p - 2}}_{L^{p'}}^p
\norm{u_n - u_m}_{L^p},
\end{multline*}
We know that $\parens{u_n}_{n \in \nat}$ is bounded in $W^{1, 2}_\alpha$ and
$W^{1, 2}_\alpha \hookrightarrow L^p$, so it suffices to show that
\begin{equation}
\label{Kirchhoff:eqn:I_5:1}
\norm{u_n - u_m}_{L^p} \xrightarrow[n, m \to \infty]{} 0.
\end{equation}
Given $n \in \nat$, let $\parens{\phi_{\lambda, n}, q_n} \in W^{1, 2} \times \complex$ be such that $u_n = \phi_{\lambda, n} + q_n G_\lambda$. In view of Proposition \ref{prelim:prop:GN}, there exists $K \in \ooi{0, \infty}$ such that
\begin{multline}
\label{Kirchhoff:eqn:I_5:2}
\norm{u_n - u_m}_{L^p}^p
\leq
K
\parens*{
	\norm{\phi_n - \phi_m}_{\dot{W}^{1, 2}}^{\frac{3 \parens{p - 2}}{2}}
	\norm{\phi_n - \phi_m}_{L^2}^{\frac{6 - p}{2}}
	+
	\frac{\abs{q_n - q_m}^p}{\lambda^{\frac{3 - p}{2}}}
}
\\
\text{for every} ~ n \in \nat.
\end{multline}
Due to the weak convergence $u_n \rightharpoonup u_\infty$ in $W^{1, 2}_\alpha$ as $n \to \infty$, we deduce that $u_n \rightharpoonup u_\infty$ in $L^2$ as $n \to \infty$ and
\begin{equation}
\label{Kirchhoff:eqn:I_5:3}
\abs{q_n - q_m} \xrightarrow[n, m \to \infty]{} 0.
\end{equation}
An application of Lemma \ref{dich:lem:2} in the case $I_\alpha = J_\alpha$ shows that
$\norm{u_\infty}_{L^2} = r$. Therefore,
$\lim_{n \to \infty} \norm{u_n - u_\infty}_{L^2} = 0$. In view of \eqref{Kirchhoff:eqn:I_5:3} and this limit, we deduce that
\begin{equation}
\label{Kirchhoff:eqn:I_5:4}
\norm{\phi_n - \phi_m}_{L^2} \xrightarrow[n, m \to \infty]{} 0.
\end{equation}
The set $\set{\phi_n}_{n \in \nat}$ is bounded in $\dot{W}^{1, 2}$, so \eqref{Kirchhoff:eqn:I_5:1} follows from \eqref{Kirchhoff:eqn:I_5:2}--\eqref{Kirchhoff:eqn:I_5:4}.
\end{proof}

\sloppy
\printbibliography
\end{document}